\documentclass[11pt,reqno]{article}
\usepackage{amsmath,amsfonts,amsthm,centernot,amssymb}
\usepackage[margin=1in]{geometry}
\usepackage[normalem]{ulem}
\usepackage[colorlinks,citecolor=blue,urlcolor=blue]{hyperref}

\usepackage[numbers]{natbib}

\numberwithin{equation}{section}
\newtheorem{Definition}{Definition}[section]
\newtheorem{Remark}{Remark}[section]
\newtheorem{Theorem}{Theorem}[section]
\newtheorem{Lemma}{Lemma}[section]
\newtheorem{Proposition}{Proposition}[section]
\newtheorem{Corollary}{Corollary}[section]
\newtheorem{Assumption}{Assumption}[section]
\newtheorem{Example}{Example}[section]

\newcommand{\be}{\begin{equation}}
\newcommand{\ee}{\end{equation}}
\newcommand{\bee}{\begin{equation*}}
\newcommand{\eee}{\end{equation*}}
\newcommand{\bi}{\begin{itemize}}
\newcommand{\ei}{\end{itemize}}

\DeclareMathOperator*{\argmax}{arg\,max}

\def \E{\mathbb{E}}

\def \N{\mathbb{N}}
\def \P{\mathbb{P}}

\def \R{\mathbb{R}}

\def \Dc{{\mathcal D}}
\def \Sc{{\mathcal S}}

\def \Pc{{\mathcal P}}
\def \Fc{{\mathcal F}}

\def \Hc{\mathcal{H}}
\def \eps{\varepsilon}
\def \tJ{\widetilde{J}}
\def \tV{\widetilde{V}}
\def \Leb{\operatorname{\texttt{Leb}}}

\title{Relaxed Equilibria for Time-Inconsistent Markov Decision Processes}
\author{Erhan Bayraktar\thanks{Department of Mathematics, University of Michigan, Ann Arbor, MI 48109-1043, USA, email: \texttt{erhan@umich.edu}. Partially supported by National Science Foundation (DMS-2106556) and by the Susan M. Smith Professorship.
}
\and
Yu-Jui Huang\thanks{
Department of Applied Mathematics, University of Colorado, Boulder, CO 80309-0526, USA, email: \texttt{yujui.huang@colorado.edu}. Partially supported by National Science Foundation (DMS-2109002).
}
\and 
Zhenhua Wang\thanks{Department of Mathematics, University of Michigan, Ann Arbor, MI 48109-1043, USA, email: \texttt{zhenhuaw@umich.edu}. 
}
\and 
Zhou Zhou\thanks{School of Mathematics and Statistics, University of Sydney, NSW 2006, Australia, email:
	\texttt{zhou.zhou@sydney.edu.au}.}
}

\date{}

\begin{document}
\maketitle	


\begin{abstract}
This paper considers an infinite-horizon Markov decision process (MDP) that allows for general non-exponential discount functions, in both discrete and continuous time. Due to the inherent time inconsistency, we look for a randomized equilibrium policy (i.e., relaxed equilibrium) in an intra-personal game between an agent's current and future selves. When we modify the MDP by entropy regularization, a relaxed equilibrium is shown to exist by a nontrivial entropy estimate. As the degree of regularization diminishes, the entropy-regularized  MDPs approximate the original MDP, which gives the general existence of a relaxed equilibrium in the limit by weak convergence arguments. As opposed to prior studies that consider only deterministic policies, 
our existence of an equilibrium does not require any convexity (or concavity) of the controlled transition probabilities and reward function. Interestingly, this benefit of considering randomized policies is unique to the time-inconsistent case.
\end{abstract}

\textbf{Keywords:} time inconsistency, non-exponential discounting, Markov decision processes, relaxed controls, entropy regularization

\section{Introduction}
For Markov decision processes (MDPs) on an infinite horizon, discounting is a key feature that allows the expected total reward to take a finite value. A widespread assumption in the literature is exponential discounting, i.e., the discount rate is constant over time. There is, however, substantial evidence against exponential discounting. In behavioral economics (see e.g., \citet{Thaler81}, \citet{LT89}, \citet{Laibson97}), it is well documented that the discount rate is empirically time-varying and many non-exponential functions have been proposed to model empirical discounting; see e.g., \citet[Remark 3.1]{HZ19}.

For the past fifteen years or so, non-exponential discounting has been seriously considered and approached in stochastic control, and the involved mathematical challenge is now well understood. In a nutshell, non-exponential discounting causes {\it time inconsistency}: an optimal control policy an agent derives today will not be optimal from the eyes of the same agent tomorrow. As a dynamically optimal policy over the entire time horizon no longer exists, \citet{Strotz55} suggests that one instead look for an {\it equilibrium policy} in an intra-personal game between one's current and future selves. 
A standard equilibrium (i.e., an equilibrium policy as a {\it deterministic} map on the state space) has been studied under non-exponential discounting in diffusion models (e.g., \citet{EL06}, \citet{Ekeland2008investment}, \citet{ekeland2012time}, \citet{yong2012time}, \citet{BKM17}, \citet{huang2018time}, \citet{HZ20},  \citet{MR4511236, BWZ23}), {\color{black}as well as in models of controlled Markov chains (e.g., \citet{BRW15}, \citet{CE16}, \citet{BRW18}, \citet{BJN2020}, \citet{HZ21}, \citet{BRW22}, \citet{BH23}). Many of the studies establish the existence of a standard equilibrium for a general state space (e.g., a Borel space), but require very specific forms of discounting (e.g., quasi-hyperbolic or hyperbolic), the reward function (e.g., monotone and supermodular), and the controlled transition probabilities (e.g., decomposable into supermodular functions); see the assumptions in \cite{BRW15, CE16, BRW18, BJN2020, BRW22}.\footnote{{\color{black}\cite{BH23} considers a finite-horizon discrete-time model, for which a recursive backward induction provides an equilibrium. This method does not work for an infinite horizon and therefore cannot be applied to our study.}} These conditions need not hold even in fairly simple examples, such that a standard equilibrium may fail to exist; see e.g., \cite[Remark 10]{HZ21}.


This motivates us to consider {\it relaxed equilibria} (i.e., {\it randomized} equilibrium policies), to possibly extend the existence of equilibria to cases where discounting, reward functions, and controlled transition probabilities are all general. To this end, it is natural to consider MDPs, which by definition considers randomized control policies (i.e., {\it relaxed controls}). 
}

In the context of reinforcement learning, \citet{AB10}, \citet{Fedus20}, and \citet{SRK22} design algorithms to compute optimal relaxed controls for MDPs under non-exponential discounting, but fail to realize that such policies are unsustainable due to time inconsistency. To the best of our knowledge, only 
the recent work \citet{JN21} recognizes the issue of time inconsistency and applies Strotz's game-theoretic approach to MDPs: 
they prove that a relaxed equilibrium exists for discrete-time MDPs for a general state space, reward function (bounded and continuous), and transition probabilities (transition densities exist and are continuous), but require the discount function to be quasi-hyperbolic. In fact, their arguments rely crucially on the specific form of a quasi-hyperbolic discount function (which is stylized and tractable) and do not allow for other kinds of non-exponential discounting.

In this paper, we accommodate {\it general} discount functions and strive to establish the existence of a relaxed equilibrium for the resulting time-inconsistent MDPs in both discrete and continuous time. 
As we can no longer rely on the form of a discount function, our approach differs largely from that in \citet{JN21}. 

Instead of working with the original MDP (i.e., \eqref{J^pi} below) directly, we consider an entropy-regularized version, where the entropy of a randomized control policy (i.e., a {\it relaxed control}) is added to the functional to be maximized; see \eqref{eq.relaxed.J} below. Taking advantage of the form of the entropy term, we characterize relaxed equilibria for the entropy-regularized MDP (which we call {\color{black}{\it regular} relaxed equilibria}) as fixed points of an operator, which takes a tractable Gibbs-measure form; see Proposition \ref{propdiscr.entropy.cha} and Corollary~\ref{coro:Psi}. By showing that the fixed-point operator continuously maps a compact domain to itself, we conclude from Brouwer's fixed-point theorem that a {\color{black}regular relaxed equilibrium} exists; 
see Theorem \ref{thm.discrete.relaxedentropy}. Note that for the operator to map a compact domain to itself, the growth of the entropy term needs to be contained appropriately, which is a known mathematical challenge. To achieve this, we assume that the reward function and controlled transition probabilities are Lipschitz in the action variable (Assumption~\ref{assume.u.lips}) and the action space $U$ fulfills a uniform cone condition (Assumption~\ref{assume.U.cone}). Then, following an argument recently proposed in \citet[Section 4.3]{HWZ22}, we get logarithmic growth of the entropy term uniformly in the state variable (Lemma~\ref{lm.1}), which facilitates the proof of Theorem~\ref{thm.discrete.relaxedentropy}. 

For the original MDP \eqref{J^pi}, relaxed equilibria can also be characterized as fixed points of an operator (Proposition~\ref{prop.chadiscrete.nonentropy}). {\color{black}However, the operator is an abstract set-valued map, which is, compared with the concrete single-valued operator under entropy regularization, much less tractable and much less promising for numerical implementation; see Remark~\ref{rem:numerical}. 
Thus,} we approximate the original MDP by a sequence of entropy-regularized MDPs, with the degree of regularization (measured by $\lambda>0$ in \eqref{eq.relaxed.J}) diminishing to zero. This yields a sequence of {\color{black}regular relaxed equilibria}, one for each entropy-regularized MDP. 
Intriguingly, the value functions corresponding to this sequence are uniformly bounded: as shown in Lemma~\ref{lm.vn.bound}, the logarithmic growth of entropy (obtained in Lemma~\ref{lm.1} for each fixed $\lambda>0$) can be made uniform across all $\lambda>0$ small enough, thereby giving a uniform bound for the value functions. Relying on this, a delicate probabilistic argument shows that the sequence of {\color{black}regular relaxed equilibria} for entropy-regularized MDPs (i.e., fixed points of tractable Gibbs-form operators) converges weakly to a fixed point of the set-valued operator for the original MDP. That is, a relaxed equilibrium for the original MDP exists; see Theorem \ref{thm.lambdaconvergence.discrete}. 

All the above, established in discrete time, can be extended to continuous time. For entropy-regularized continuous-time MDP, {\color{black}regular relaxed equilibria} can again be characterized as fixed points of an operator of the Gibbs-measure form; see Proposition \ref{propconti.entropy.cha} and Corollary~\ref{coro:Psi'}. The existence of a {\color{black}regular relaxed equilibrium} is  accordingly established in Theorem~\ref{thm1}. Finally, we approximate a continuous-time MDP by a sequence of its entropy-regularized versions, with diminishing degree of regularization, and find that the sequence of {\color{black}regular relaxed equilibria} (one for each entropy-regularized MDPs) converges to a relaxed equilibrium for the original MDP; see Theorem~\ref{thm.lambdaconvergence.continuous}.  Note that our continuous-time results are reconciled with those in discrete time: as shown in Theorem~\ref{thm.discts.convergence}, when the time step tends to zero, {\color{black}(regular)} relaxed equilibria in discrete time converge to one in continuous time. 

Our mathematical setup generalizes that in \citet{HZ21}, where controlled Markov chains (in both discrete and continuous time) are considered under non-exponential discounting. Specifically, one chooses transition probabilities directly (i.e., an action is simply a vector of transition probabilities) in \citet{HZ21}, while we allow for general actions that may affect transition probabilities in a much more subtle way; {\color{black}see Remark~\ref{rem:more general p^u} for details.} 
In addition, \citet{HZ21} focuses on deterministic policies (i.e., standard controls), while we include randomized ones (i.e., relaxed controls). A key result in \citet{HZ21} states that a standard equilibrium exists, under a suitable condition on the transition probabilities and reward function; moreover, if such a condition fails, a standard equilibrium may not exist in general. This crucially explains the need of relaxed controls: unlike standard equilibria, a relaxed equilibrium generally exists, as proved in Theorems~\ref{thm.lambdaconvergence.discrete} and \ref{thm.lambdaconvergence.continuous}. Hence, in the face of general controlled transition probabilities and reward function, relaxed equilibria should certainly be considered; see Remark~\ref{rem:relaxed needed} and the discussion below it. 

Interestingly, the need of relaxed controls is unique to the time-inconsistent case. In the time-consistent case of exponential discounting, even if one considers relaxed controls, the optimal value achieved by a relaxed control can always be achieved alternatively by a standard control---it is then necessary to consider only standard controls; see Remarks~\ref{rm.discrete} and \ref{rm.conti}. 

Let us stress that an entropy-regularized MDP, besides serving as a powerful approximation tool, has it own meaning and application. As introduced in \citet{Ziebart08, Fox16} and clearly explained in \citet{WZZ20}, an entropy-regularized MDP  encodes the tradeoff between exploitation and exploration in reinforcement learning, with the exploration part represented by the added entropy term. This line of research has so far focused on time-consistent cases. The only exception we know of is \citet{DDJ23}, where a relaxed equilibrium is found for a mean-variance portfolio selection problem. Our results contribute to the burgeoning area of reinforcement learning under time inconsistency: the {\color{black}regular relaxed equilibrium} we found represents a learning policy under time inconsistency induced by non-exponential discounting; see the discussion below Theorem~\ref{thm.lambdaconvergence.discrete} and Remark~\ref{rem:RL}.  

The paper is organized as follows. Section \ref{sec:discrete} introduces a time-inconsistent MDP in discrete time that accommodates non-exponential discounting; an entropy-regularized version is also defined. We prove the existence of {\color{black}regular relaxed equilibria} for entropy-regularized MDPs (Section \ref{subsec.discrete.entropy}) and obtain the  existence of relaxed equilibria for the original MDP by an approximation argument (Section \ref{subsec.discrete.noentropy}).  Section~\ref{subsec.cts.entropy} (resp.\ Section \ref{subsec.cts.noentropy}) extends results in Section \ref{subsec.discrete.entropy} (resp.\ Section \ref{subsec.discrete.noentropy}) to continuous time. Section \ref{subsec:relaxed needed} discusses when the use of relaxed controls is necessary. Section~\ref{sec:disctsconverges} shows that discrete-time {\color{black}(regular)} relaxed equilibria converge to a {\color{black}(regular)} relaxed equilibrium in continuous time, as time step tends to zero. The appendix collects (longer) proofs. 


\section{Relaxed Equilibria in Discrete Time}\label{sec:discrete}
Set $\N_0 := \N\cup \{0\}$. Let $U\subset \R^\ell$, for a fixed $\ell\in\N$, be a compact action space with $\Leb(U)>0$. Consider a probability space $(\Omega,\Fc,\P)$ that supports a discrete-time Markov process $(X_t)_{t\in \N_0}$ that takes values in $\Sc=\{1,2,...,d\}$ for some $d\in\N$ and is controlled by a $U$-valued process $\alpha$. Assume that the dynamics of $X$ is time-homogeneous, i.e., for any $t\in \N_0$, $\P(X_{t+1}=j\mid X_t=i, \alpha_t=u) = \P(X_1= j\mid X_0=i, \alpha_0=u)$ for all $i,j\in\Sc$ and $u\in U$. For any action $u\in U$, we will denote by $p^u$ the associated transition matrix of $X$ (i.e., $p^u_{ij} := \P(X_1=j\mid X_0=i, \alpha_0=u)$ for all $i,j\in\Sc$) and write $p^u_i$ for the $i^{th}$ row of $p^u$ for all $i\in\Sc$. 

We call $\alpha:\Sc\to U$ a  {\it standard feedback control} and denote by $p^\alpha$ the transition matrix induced by $\alpha$. Specifically, the $i^{th}$ row of $p^\alpha$ (denoted by $p^\alpha_i$) is given by $p^\alpha_i = p^{\alpha(i)}_i$ for all $i\in\Sc$. 
Let $\Pc(U)$ (resp.\ $\Dc(U)$) denote the set of all probability measures (resp.\ density functions) on $U$. We call $\pi:\Sc\to \Pc(U)$ a {\it relaxed feedback control}. Note that a standard feedback control $\alpha$ can be viewed as a relaxed feedback control $\pi$ by taking $\pi(i)$ to be the Dirac measure concentrated on $\alpha(i)\in U$ for all $i\in\Sc$. We denote by $\Pi$ the set of all relaxed feedback controls.


Given $\pi\in\Pi$, the dynamics of $X=X^\pi$ is determined as follows. At any time $t\in\N_0$, given that $X_t =i \in \Sc$, we sample $u\in U$ according to the probability measure $\pi(i)\in\Pc(U)$. The realization of $X_{t+1}$ is then governed by the transition probabilities $p^{u}_i = \{p^u_{ij}\}_{j\in\Sc}$. 

\begin{Remark}\label{rem:same law}
	Given $\pi\in\Pi$, consider a Markov chain $\bar X^\pi$ whose transition matrix $p^\pi$ is given by  $p^\pi_{ij} := \int_U p^u_{ij} (\pi(i))(du)$ for all $i,j\in\Sc$. Note that $\bar X^\pi$ and $X^\pi$ are different stochastic processes that share the same law. That is, the paths (i.e., realizations) taken by $\bar X^\pi$ and $X^\pi$ are generally different, but the probability of which path will be taken is identical. 
\end{Remark}

Given a reward function $f:[0,\infty)\times\Sc\times U\to \R$, we define $f^\mu(t, i):= \int_{U} f(t, i, u)\mu(du)$ for all $\mu\in \Pc(U)$. For any $\pi\in\Pi$, the corresponding value function is given by 
\begin{equation}\label{J^pi}
	J^\pi(i) := \E_i\bigg[ \sum_{k=0}^\infty f^{\pi(X^\pi_k)}\left(k,X^\pi_k\right)\bigg],\quad \forall i\in \Sc.
\end{equation}

\begin{Remark}\label{rem:delta}
	The $t$ variable in $f(t,i,u)$ does not represent ``real calendar time'' but ``time difference,'' that is, the difference between the current time and the time of a future reward; see \citet[Remark 1]{HZ21}. A typical example is 
	\begin{equation}\label{discount problem}
		f(t,i,u)=\delta(t) g(i,u), 
	\end{equation}
	where $g:\Sc\times U\to \R$ assigns a reward based on the current state $i$ and the action $u$ employed and $\delta:[0,\infty)\to [0,1]$ is a discount function, assumed to be nonincreasing with $\delta(0)=1$. 
\end{Remark}

In this paper, we will investigate an entropy-regularized version of $J^\pi(i)$. To this end, let us first introduce the notion of a regular relaxed feedback control. 

\begin{Definition}\label{def:rrc}
	A relaxed feedback control $\pi\in\Pi$ is {\it regular} if for each $i\in \Sc$, there exists $\rho_i\in\Dc(U)$ such that 
	$(\pi(i))(A) = \int_A \rho_i(u) du$ for all Borel $A\subseteq U$ and its  
	Shannon differential entropy satisfies $\Hc(\rho_i)>-\infty$, where 
	\begin{equation}\label{H}
		\Hc(\rho) := - \int_{U} \ln(\rho(u)) \rho(u) du,\quad \forall \rho\in\Dc(U).
	\end{equation}
	We denote by $\Pi_r$ the set of all regular feedback relaxed controls. Given $\pi\in\Pi_r$, we will often identify $\pi(i)\in\Pc(U)$ with its density $\rho_i\in\Dc(U)$ and write ``$\pi\in\Pi_r$'' and ``$\{\rho_i\}_{i\in\Sc}\in\Pi_r$'' interchangeably. 
\end{Definition}

Now, for any $\lambda> 0$ and $\pi\in\Pi_r$, consider 
\be\label{eq.relaxed.J} 
J^\pi_\lambda(i):= \E_i\bigg[ \sum_{k=0}^\infty \left( f^{\pi({\color{black}X^\pi_k})}(k,X^\pi_k) + \lambda  \delta(k)\Hc\big({\pi({\color{black}X^\pi_k})}\big) \right) \bigg],\quad \forall i\in \Sc,
\ee 
where {\color{black}$\delta:[0,\infty)\to [0,1]$ is a discount function, assumed to be nonincreasing with $\delta(0)=1$.} 
To ensure that $J^\pi$ in \eqref{J^pi} and $J^\pi_\lambda$ in \eqref{eq.relaxed.J} are well-defined, we impose the following.

\begin{Assumption}\label{assume.f.p}
	$M:= \sum_{t=0}^\infty \left( \sup_{i\in \Sc, u\in U} |f(t,i,u)|+\delta(t)\right)<\infty$. 
\end{Assumption}

By Assumption~\ref{assume.f.p}, $J^\pi$ is clearly finitely valued for all $\pi\in\Pi$. On the other hand, under our assumption $0<\Leb(U)<\infty$, consider the uniform density $\nu\in\Dc(U)$ given by $\nu(u):=1/\Leb(U)$ for all $u\in U$. For any $\rho\in \Dc(U)$, we compute the Kullback-Leibler divergence
\begin{equation}\label{eq.kl.lbound}
	0\leq D_{KL}(\rho\|\nu) :=\int_U \rho\ln\left(\frac{\rho}{\nu}\right)du=\int_U \rho\ln \rho du+\ln (\Leb(U)),
\end{equation}
which gives $\Hc(\rho) = -\int_{U} \ln(\rho(u))\rho(u)du\leq \ln(\Leb(U))<\infty$ for all $\rho\in \Dc(U)$. This, along with Definition~\ref{def:rrc}, indicates that for any $\pi\in\Pi_r$, 
\begin{equation}\label{H finite}
	|\Hc({\pi(i)})| < \infty\quad \hbox{for all}\ i\in\Sc.   
\end{equation}
By Assumption~\ref{assume.f.p}, \eqref{H finite}, and $\Sc$ being a finite set, $J^\pi_\lambda$ is finitely valued for all $\pi\in\Pi_r$. 


An agent who aims to maximize $J^\pi(i)$ over all $\pi\in \Pi$ may run into the issue of time inconsistency. 
{\color{black} Specifically, given that $X_0=i\in\Sc$, the time-0 agent's problem is $\sup_{\pi\in\Pi} J^\pi(i)$. At a later time $t>0$ with $X_t = j\neq i$, the time-$t$ agent's problem is $\sup_{\pi\in\Pi}J^\pi(j)$. As the two problems $\sup_{\pi\in\Pi} J^\pi(i)$ and $\sup_{\pi\in\Pi} J^\pi(j)$ need not share the same optimal control $\pi^*\in\Pi$, time-inconsistency may arise.}
Such inconsistency results from the ``time difference'' variable $t$ in either the reward function $f$ or the discount function $\delta$; see Remark~\ref{rem:delta}. In the typical setup \eqref{discount problem}, it is well-known that the optimization problems are time-consistent with $\delta(t) = e^{-\beta t}$ for some $\beta>0$ (i.e., the case of exponential discounting) but time-inconsistent in general. 


As proposed in \citet{Strotz55}, a sensible reaction to time inconsistency is to take future selves' disobedience into account and choose the best present action in response to that. Assuming that all future selves reason in the same way, the agent searches for a (subgame perfect) equilibrium strategy from which no future self has an incentive to deviate. To formulate such an equilibrium strategy, we introduce, for any $\pi, \pi'\in \Pi$, the concatenation of $\pi'$ and $\pi$ at time $1$, denoted by $\pi'\otimes_1\pi\in\Pi$.  Using $\pi'\otimes_1\pi\in\Pi$ 
means that the evolution of $X$ is governed first by $\pi'$ at time 0 and then by $\pi$ from time 1 onward. {\color{black} That is, at time 0, given that $X_0 =i \in \Sc$, we sample $u\in U$ according to the measure $\pi'(i)\in\Pc(U)$ and the realization of $X_{1}$ is then governed by the transition probabilities $p^{u}_i = \{p^u_{ij}\}_{j\in\Sc}$. At any time $t\ge 1$, given that $X_t=j\in\Sc$, we sample $\bar u\in U$ according to the measure $\pi(j)\in\Pc(U)$ and the realization of $X_{t+1}$ is then governed by $p^{\bar u}_j = \{p^{\bar u}_{jk}\}_{k\in\Sc}$. 
} 


\begin{Definition}\label{def.relaxedequi.discrete}
	Given $\lambda> 0$, we say $\pi\in\Pi_r$ is a {\color{black}regular relaxed equilibrium} (for \eqref{eq.relaxed.J}) 
	if for any $ i\in \Sc$,
	\be\label{eq.def.equidiscrete} 
	J^{\pi'\otimes_1 \pi}_\lambda(i)- J^{\pi}_\lambda(i)\leq 0,\quad \forall \pi'\in \Pi_r. 
	\ee 
	Similarly, we say $\pi\in\Pi$ is a relaxed equilibrium (for \eqref{J^pi}) if for any $ i\in \Sc$, 
	\be\label{eq.def.equidiscrete'} 
	J^{\pi'\otimes_1 \pi}(i)- J^{\pi}(i)\leq 0,\quad \forall \pi'\in \Pi. 
	\ee 
\end{Definition}


\subsection{Existence of {\color{black}Regular Relaxed Equilibria} for Entropy-Regularized MDP \eqref{eq.relaxed.J}}\label{subsec.discrete.entropy}
Let us first focus on the existence of regular relaxed equilibria (for \eqref{eq.relaxed.J}). 
Fix $\lambda>0$. For any $\pi\in\Pi_r$, let us introduce the auxiliary value function 
\be\label{eq.relaxed.V.lambda} 
V_\lambda^\pi(i):= \E_i \bigg[ \sum_{k=0}^\infty \left( f^{\pi\left( {\color{black}X^\pi_k} \right)}(1+k,X^\pi_k)+\lambda\delta(1+k) \Hc\big({\pi({\color{black}X^\pi_k})}\big) \right) \bigg], \quad \forall i\in\Sc. 
\ee 
For convenience, we will commonly write $V^\pi_\lambda$ as a vector in $\R^d$, i.e., $V^\pi_\lambda = (V^\pi_\lambda(1), V^\pi_\lambda(2),...,V^\pi_\lambda(d))$. 
{\color{black}
To understand the meaning of $V^\pi_\lambda$, we need to take a closer look at $J^\pi_\lambda(i)$ in \eqref{eq.relaxed.J}. Given that $X_0=i\in\Sc$, notice that the first term in the summation of \eqref{eq.relaxed.J} is no longer random and can be computed immediately. The remaining terms in the summation are still random and their expectation depends on the realization of $X^\pi_1$. Specifically,
\be\label{V_lambda meaning} 
\E_i\bigg[ \sum_{k=1}^\infty \left( f^{\pi(X^\pi_k)}(k,X^\pi_k) + \lambda  \delta(k)\Hc\big({\pi(X^\pi_k)}\big) \right)\bigg|\ X^\pi_1=j \bigg]= V^\pi_\lambda(j),\quad \forall j\in \Sc,
\ee 
where the equality follows from the Markov property of $X^\pi$. That is, $V^\pi_\lambda(j)$ is the expectation of future rewards plus entropy conditioned on $X_1^\pi=j$. 

Now, for the problem \eqref{eq.relaxed.J}, given $X_0=i\in\Sc$ and that all future selves will follow a relaxed control $\pi\in\Pi_r$, the agent at time 0 intends to find her best strategy (denoted by $\pi'\in\Pi_r$) in response to that. Observe from \eqref{eq.relaxed.J} and \eqref{V_lambda meaning} that
\begin{align*}
J_\lambda^{\pi'\otimes_1\pi}(i) &= f^{\pi'(i)}(0,i)+\lambda \Hc(\pi'(i))+\E_i[V^{\pi}_\lambda(X^{\pi'}_1)] \\
&=  \int_U \Big( f(0,i,u) -\lambda \ln((\pi'(i))(u))+ p^u_i\cdot V^\pi_\lambda \Big) (\pi'(i))(u)du.
\end{align*}
Hence, the best strategy $\pi'\in\Pi_r$ for the time-0 agent should satisfy
\begin{equation}\label{unique argmax}
\pi'(i)\in  \argmax_{\rho\in \Dc(U)}   \int_U \Big( f(0,i,u) -\lambda \ln(\rho(u))+ p^u_i\cdot V^\pi_\lambda \Big) \rho(u)du,\quad\forall i\in \Sc.
\end{equation}
Note that the set of maximizers on the right-hand side is a singleton and its unique element takes the explicit form 
	\[
	\rho^*_i(u) := \frac{e^{\frac{1}{\lambda} \big(f(0,i,u)+ p^u_i\cdot V^\pi_\lambda\big)} }{\int_{U} e^{\frac{1}{\lambda} \big(f(0,i,u)+ p^u_i\cdot V^\pi_\lambda\big)} du},\quad u\in U.
	\]
As a result, by introducing a functional $\Gamma_\lambda: \R^d\times\Sc\to \Dc(U)$ defined by
\begin{equation}\label{Gamma_lambda}
	u\mapsto\Gamma_\lambda(y,i) (u) : = \frac{e^{\frac{1}{\lambda} \left(f(0,i,u) +p^u_i\cdot y\right)} }{\int_{U} e^{\frac{1}{\lambda} \left(f(0,i,v) +p^v_i\cdot y\right)} dv}\in \Dc(U),\quad \forall (y,i)\in \R^d\times\Sc,
\end{equation}
we can re-write \eqref{unique argmax} as 
\be\label{optimal pi'}
\pi'(i) = \Gamma_\lambda(V^\pi_\lambda, i)(\cdot)\in\Dc(U),\quad \forall i\in\Sc. 
\ee

We have argued so far that $\pi'\in\Pi_r$ in \eqref{optimal pi'} is the time-0 agent's best response to her future selves using $\pi\in\Pi_r$. If it happens that the best response $\pi'\in\Pi_r$ coincides with future selves' strategy $\pi\in\Pi_r$ (i.e., $ \Gamma_\lambda(V^\pi_\lambda, \cdot)=\pi$), then $\pi\in\Pi_r$ should be viewed as an equilibrium among the current and future selves, as it is a strategy that will be upheld over time. This motivates us to define an operator $\Phi_\lambda:\Pi_r\to\Pi_r$ by 
\be\label{Phi_lambda}
\Phi_\lambda(\pi) := \Gamma_\lambda(V^\pi_\lambda,\cdot), 
\ee
and we conjecture that a fixed point of $\Phi_\lambda$ is a regular relaxed equilibrium for \eqref{eq.relaxed.J}. In addition, as $\Gamma_\lambda(V^\pi_\lambda, \cdot)=\pi$ consequently yields $V^{\Gamma_\lambda(V^\pi_\lambda, \cdot)}_\lambda=V^\pi_\lambda$, we also define an operator $\Psi_\lambda:\R^d\to\R^d$ by 
\[
\Psi_\lambda(y) := V^{\Gamma_\lambda (y,\cdot)}_\lambda, 
\]
and conjecture that a fixed point of $\Psi_\lambda$ must equal $V^\pi_\lambda$ for some regular relaxed equilibrium $\pi\in\Pi_r$.


The next two results show that our conjectures are correct. 
}

\begin{Proposition}\label{propdiscr.entropy.cha}
	Let Assumption \ref{assume.f.p} hold.  For any $\lambda>0$, 
	\[
	\hbox{$\pi\in \Pi_r$ is a {\color{black}regular relaxed equilibrium}}\  \iff\  \Phi_\lambda(\pi)=\pi.
	\]
\end{Proposition}

\begin{proof}
	Given $\pi,\pi'\in \Pi_r$, since $J_\lambda^{\pi'\otimes_1\pi}(i)=f^{\pi'(i)}(0,i)+\lambda \Hc(\pi'(i))+\E_i[V^{\pi}_\lambda(X^{\pi'}_1)$] for all $i\in\Sc$, a direct calculation shows
	\begin{align*}
		J_\lambda^{\pi'\otimes_1 \pi}(i)- J_\lambda^{\pi}(i) &=f^{\pi'(i)}(0,i)+\lambda \Hc(\pi'(i))+\int_{U} (p^u_i\cdot V^\pi_\lambda) (\pi'(i))(u) du\\
		&\hspace{0.4in}-\left(f^{\pi(i)}(0,i)+\lambda \Hc(\pi(i))+ \int_{U} (p^u_i\cdot V^\pi_\lambda) (\pi(i))(u)du\right),\quad \forall i\in\Sc.
	\end{align*}
	It follows that \eqref{eq.def.equidiscrete} holds (i.e., $\pi$ is a {\color{black}regular relaxed equilibrium}) if and only if ${\pi(i)}\in\Dc(U)$ fulfills
	\be\label{eq.discrete.lm} 
	\begin{aligned}
		{\pi(i)} &\in  \argmax_{\rho\in \Dc(U)}   \int_U \Big( f(0,i,u) -\lambda \ln(\rho(u))+ p^u_i\cdot V^\pi_\lambda \Big) \rho(u)du,\quad\forall i\in \Sc.
	\end{aligned}
	\ee 
	By the same arguments in \eqref{unique argmax}-\eqref{optimal pi'}, we can express \eqref{eq.discrete.lm} equivalently as  
	${\pi(i)} = \Gamma_\lambda(V^\pi_\lambda, i)$ for all $i\in\Sc$, which amounts to $\pi = \Phi_\lambda(\pi)$. 
\end{proof}

\begin{Corollary}\label{coro:Psi}
	Let Assumption \ref{assume.f.p} hold and $\lambda>0$. For any $y\in\R^d$, 
	\[
	\hbox{$y= V^\pi_\lambda$ for some {\color{black}regular relaxed equilibrium} $\pi\in\Pi_r$}\ \iff\ \Psi_\lambda(y)=y. 
	\]
	In particular, $\Psi_\lambda(y)=y$ implies that $\Gamma_\lambda(y,\cdot)\in\Pi_r$ is a {\color{black}regular relaxed equilibrium}. 
\end{Corollary}

\begin{proof}
	If $y=V^\pi_\lambda$ for a {\color{black}regular relaxed equilibrium} $\pi\in\Pi_r$, then by Proposition~\ref{propdiscr.entropy.cha}, $\pi=\Phi_\lambda(\pi)= \Gamma_\lambda(V^{\pi}_\lambda,\cdot)=\Gamma_\lambda(y,\cdot)$, which implies $y=V^\pi_\lambda=V^{\Gamma_\lambda(y,\cdot)}_\lambda=\Psi_\lambda(y)$. Conversely, if $y= \Psi_\lambda(y)=V^{\Gamma_\lambda(y,\cdot)}_\lambda$, set $\pi:= \Gamma_\lambda(y,\cdot)\in\Pi_r$. Then, we have $y=V^\pi_\lambda$ and thus $\pi = \Gamma_\lambda(y,\cdot)=\Gamma_\lambda(V^{\pi}_\lambda,\cdot)=\Phi_\lambda(\pi)$. By Proposition~\ref{propdiscr.entropy.cha}, this implies that $\pi$ is a {\color{black}regular relaxed equilibrium}.
\end{proof}

To properly control the entropy of the fixed-point operator $\Phi_\lambda$ in \eqref{Phi_lambda}, we need the following two assumptions.

\begin{Assumption}\label{assume.u.lips} The maps $u\mapsto f(t,i,u)$ and $u\mapsto p^u_i$ are Lipschitz, uniformly in $(t,i)$, i.e., 
	\be\label{eq.assum.lip} 
	\Theta := \sup_{t\in\N_0, i\in \Sc}\sup_{u_1,u_2\in U, u_1\neq n_2}\left\{\frac{|f(i,t,u_1)-f(i,t,u_2)|}{|u_1-u_2|}+\frac{|p^{u_1}_i-p^{u_2}_i|}{|u_1-u_2|}\right\} <\infty.
	\ee 
\end{Assumption}

We will also assume that the action space $U\subset \R^\ell$ fulfills a {\it uniform cone condition}. To properly state the condition, 
for any $\iota\in [0,\pi/2]$, we note that
$$
\Delta_\iota:= \{u=(u_1,...,u_{\ell})\in\R^{\ell} : u_1^2+...+u_{\ell-1}^2\leq \tan^2(\iota) u_{{\ell}}^2\} 
$$
is a cone with vertex, axis, and angle being $0\in\R^{d-1}$, $u_1=u_2=...=u_{{\ell}-1}=0$, and $\iota$, respectively. Now, given $u\in\R^{\ell}$, the region obtained by a rotation of $u+\Delta_\iota$ in $\R^{\ell}$ about $u$ will be called {\it a cone with vertex $u$ and angle $\iota$}. 

\begin{Assumption}\label{assume.U.cone}
	When $\ell>1$, there exists $\vartheta>0$ and $\iota\in (0, \pi/2]$ such that for any 
	$u\in U$, there is a cone with vertex $u$ and angle $\iota$ (denoted by $\operatorname{cone}(u,\iota)$) that satisfies $\left(\operatorname{cone}(u,\iota)\cap B_{\vartheta}(u)\right)\subseteq U$. When $\ell=1$, there exists $\vartheta>0$ such that for any $u\in U$, either $[u-\vartheta,u]$ or $[u, u+\vartheta]$ is contained in $U$.
\end{Assumption}

\begin{Remark}
	Assumption~\ref{assume.U.cone} states that a cone with a fixed size (determined by slant height $\vartheta$ and angle $\alpha$) can be attached to any $u\in U$ (i.e., taking $u$ as its vertex) such that the cone is contained entirely in $U$.
	This readily covers all polyhedrons and ellipsoids in $\R^{\ell}$.
\end{Remark}

We can now establish a key estimate of the entropy of the fixed-point operator $\Phi_\lambda$ in \eqref{Phi_lambda}, whose proof is relegated to Appendix~\ref{subsec:proof of lm.1}.

\begin{Lemma}\label{lm.1}
	Let Assumptions \ref{assume.u.lips} and \ref{assume.U.cone} hold. Then, 	
	\be\label{eq:supH esti} 
	\sup_{i\in\Sc}|\Hc(\Gamma_\lambda(y,i)| 
	\leq \phi(|y|), \quad \forall y \in \R^d, 
	\ee 
	where $\phi:\R_+\to \R_+$ is defined by  $\phi(z):=\kappa+\ell\ln(1+z)$, with $\kappa>0$ depending on only $\ell$, $\lambda$, $\Leb(U)$, $\iota$, $\vartheta$, and $\Theta$. 
	{\color{black} Moreover, for $\lambda>0$ small enough, \eqref{eq:supH esti} can be improved to 
		\be\label{eq:supH esti'}
		\sup_{i\in \Sc} \left|\Hc(\Gamma_\lambda(y,i))\right|\le \varphi(\lambda,y):=  \kappa_1 +  \kappa_2|\ln \lambda|+\ell\ln (1+|y|),\quad \forall y\in\R^d,
		\ee
		where $\kappa_1,\kappa_2>0$ depend on $\ell$, $\Leb(U)$, $\iota$, $\vartheta$, and $\Theta$, but not on $\lambda$. }
\end{Lemma}


Now, we are ready to present the existence of a regular relaxed equilibrium, whose proof is relegated to Appendix~\ref{subsec:proof of thm.discrete.relaxedentropy}.

\begin{Theorem}\label{thm.discrete.relaxedentropy}
	Let Assumptions \ref{assume.f.p}, \ref{assume.u.lips}, and
	\ref{assume.U.cone} hold. For any $\lambda>0$, there exists $y\in\R^d$ such that $\Psi_\lambda(y)=y$. Hence, $\Gamma_\lambda(y,\cdot)\in\Pi_r$ is a {\color{black}regular relaxed equilibrium} for \eqref{eq.relaxed.J}. 
\end{Theorem}

Interestingly, as the degree of regularization tends to zero (i.e., $\lambda\downarrow 0$ in \eqref{eq.relaxed.J}), the next result shows that the values generated by the corresponding {\color{black}regular relaxed equilibria} (whose existence is guaranteed by Theorem~\ref{thm.discrete.relaxedentropy}) are uniformly bounded, thanks to the entropy estimate in Lemma~\ref{lm.1}. Its proof is relegated to Appendix~\ref{subsec:proof of lm.vn.bound}. 

\begin{Lemma}\label{lm.vn.bound}
	Let Assumptions \ref{assume.f.p}, \ref{assume.u.lips}, and \ref{assume.U.cone} hold. Given $\{\lambda_n\}_{n\in \N}$ in $(0,1]$ with $\lambda_n\downarrow 0$, consider $\{y^n\}_{n\in \N}$ in $\R^d$ such that $y^n=\Psi_{\lambda_n}(y^n)$. Then, $\sup_{n\in \N} |y^n|<\infty$. 
\end{Lemma} 

\subsection{Existence of Relaxed Equilibria for MDP \eqref{J^pi}}\label{subsec.discrete.noentropy}
Now, we move on to prove the existence of a relaxed equilibrium for the original MDP \eqref{J^pi}. Similarly to \eqref{eq.relaxed.V.lambda}, 
for any $\pi\in\Pi$, we introduce the auxiliary value function 
\be\label{eq.relaxed.V} 
V^\pi(i):= \E_i \bigg[ \sum_{k=0}^\infty f^{\pi\left( {\color{black}X^\pi_k} \right)}\left( 1+k,X^\pi_k \right) \bigg], \quad \forall i\in\Sc. 
\ee 
For convenience, we will commonly write $V^\pi$ as a vector in $\R^d$, i.e., $V^\pi = (V^\pi(1), V^\pi(2),...,V^\pi(d))$. 
Suppose that $u\mapsto f(t,i,u)$ and $u\mapsto p^u_i$ are continuous. As $U$ is compact,  for any $(y,i)\in\R^d\times\Sc$,
\be\label{eq.discrete.supp}  
E(y,i) := \argmax_{u\in U} \left\{ f(0,i, u)+ p^u_i\cdot y  \right\}\subseteq U
\ee
is nonempty and closed. For any $y\in\R^d$, consider the collection 
\begin{equation}\label{Gamma}
	\Gamma(y) := \left\{\pi\in\Pi : \operatorname{supp}(\pi(i))\subseteq E(y,i),\ \forall i\in \Sc \right\},
\end{equation}
where $\operatorname{supp}(\rho)$ denotes the support of $\rho\in\Dc(U)$. We can then define a set-valued operator $\Psi:\R^d\to2^{\R^d}$ by 
\begin{equation}\label{Psi}
	\Psi(y) := \{V^{\pi}:\pi\in\Gamma(y)\}\subseteq \R^d.  
\end{equation}
Moreover, we can also define a set-valued operator $\Phi:\Pi\to2^\Pi$ by 
\be\label{Phi}
\Phi(\pi) := \Gamma(V^\pi)\subseteq\Pi. 
\ee
Let us provide the following characterizations of relaxed equilibrium for \eqref{J^pi}. 

\begin{Proposition}\label{prop.chadiscrete.nonentropy}
	Let Assumption \ref{assume.f.p} hold and the maps $u\mapsto f(t,i,u)$ and $u\mapsto p^u_i$ be continuous. Then, 
	\begin{equation}\label{pi in Phi}
		\hbox{$\pi\in \Pi$ is a relaxed equilibrium}\  \iff\  \pi\in \Phi(\pi).
	\end{equation}
	Moreover, for any $y\in\R^d$, 
	\[
	\hbox{$y= V^\pi$ for some relaxed equilibrium $\pi\in\Pi$}\ \iff\ y\in\Psi(y). 
	\]
\end{Proposition}

\begin{proof}
	By the same argument in the proof of Proposition \ref{propdiscr.entropy.cha} (while ignoring the term $\lambda \Hc(\cdot)$ therein), we observe that $\pi\in\Pi$ is a relaxed equilibrium if and only if $\pi(i)\in\Pc(U)$ fulfills 
	$$
	\pi(i)\in \argmax_{\mu\in \Pc(U)} \int_U \left( f(0,i,u) + p^u_i\cdot V^\pi \right) \mu(du),\quad \forall i\in\Sc,
	$$ 
	which is equivalent to $\operatorname{supp}(\pi(i))\subseteq E(V^\pi,i)$ for all $i\in\Sc$, i.e., $\pi\in\Gamma(V^\pi)=\Phi(\pi)$. 
	
	For any $y\in \R^d$ such that $y\in \Psi(y)$, in view of \eqref{Psi}, $y=V^\pi$ for some $\pi\in \Gamma(y)$. It follows that $\pi \in \Gamma(y) = \Gamma(V^\pi)=\Phi(\pi)$. By \eqref{pi in Phi}, this implies that $\pi$ is a relaxed equilibrium. Conversely, suppose that $y=V^\pi$ for some relaxed equilibrium $\pi\in\Pi$. By \eqref{pi in Phi}, $\pi\in \Phi(\pi) = \Gamma(V^\pi) = \Gamma(y)$. With $y=V^\pi$ and $\pi\in\Gamma(y)$, we immediately conclude $y\in\Psi(y)$.
\end{proof}

With the aid of Lemma~\ref{lm.vn.bound} and Proposition~\ref{prop.chadiscrete.nonentropy}, we are ready to present the existence of a relaxed equilibrium for \eqref{J^pi}, by approximating \eqref{J^pi} using a sequence of entropy-regularized MDPs. The detailed proof is relegated to Appendix~\ref{subsec:proof of thm.lambdaconvergence.discrete}.


\begin{Theorem}\label{thm.lambdaconvergence.discrete}
	Let Assumptions \ref{assume.f.p}, \ref{assume.u.lips}, and \ref{assume.U.cone} hold. Then, a relaxed equilibrium for \eqref{J^pi} exists.
\end{Theorem}

{\color{black}
This paper mainly takes the entropy-regularized MDP \eqref{eq.relaxed.J} as a powerful approximation tool for the original MDP \eqref{J^pi}, as shown in the proof of Theorem~\ref{thm.lambdaconvergence.discrete} (see Appendix~\ref{subsec:proof of thm.lambdaconvergence.discrete}). 
Yet, \eqref{eq.relaxed.J} does have its own meaning and application. In the context of reinforcement learning (RL), an agent does not know the model perfectly (e.g., the MDP's transition probabilities may not be precisely known). She then chooses her control actions for two different purposes---one is to enlarge her cumulative payoff based on her present knowledge of the model (i.e., ``exploitation''); the other is to obtain more information about the model based on the observed MDP evolution (i.e., ``exploration''). To enhance ``exploration,'' the agent randomizes her control actions (i.e., chooses a relaxed control) to more efficiently infer the model (from the more diverse MDP evolution), and the amount of information gained can be measured by Shannon's entropy of the randomization. This corresponds to the second term in the expectation of \eqref{eq.relaxed.J}. The chosen relaxed control also needs to serve the ``exploitation'' purpose, 
which corresponds to the first term in the expectation of \eqref{eq.relaxed.J}. 

\begin{Remark}\label{rem:RL}
For typical RL without time inconsistency (e.g., $\delta(t) = e^{-\beta t}$ for some $\beta>0$ in \eqref{eq.relaxed.J}), the agent's goal is to find a {\it relaxed control} that maximizes \eqref{eq.relaxed.J}, thereby striking a balance between ``exploitation'' and ``exploration'' (see e.g., \citet{Ziebart08, Fox16, WZZ20}).   
When $\delta$ is a general discount function, the agent also needs to tackle the issue of time inconsistency. 
That is, besides striking a balance between ``exploitation'' and ``exploration,'' she also wants to maintain the balance among all disobedient future selves. The agent then aims for a relaxed control for \eqref{eq.relaxed.J} that can be upheld by all current and future selves, i.e., a {\it regular relaxed equilibrium} in Definition~\ref{def.relaxedequi.discrete}. Theorem~\ref{thm.discrete.relaxedentropy} asserts that such a desired RL policy exists. 
\end{Remark}
}


{\color{black}
\begin{Remark}\label{rem:limit not regular}
For $\lambda>0$, let $\pi_\lambda\in\Pi_r$ be a regular relaxed equilibrium for \eqref{eq.relaxed.J}. 
{Given $i\in\Sc$, $\pi_\lambda(i)\in\Pc(U)$ admits a density function for all $\lambda>0$ (as $\pi_\lambda$ is regular; see Definition~\ref{def:rrc}). However, as $\lambda\downarrow 0$, the weak limit $\pi^*(i)$ of $\{\pi_\lambda(i)\}_{n\in\N}$ may not admit a density function. This is why in Theorem~\ref{thm.lambdaconvergence.discrete}, we get only a relaxed equilibrium,  which is not necessarily regular, for \eqref{J^pi}. }
\end{Remark}
}

{\color{black}
\begin{Remark}\label{rem:numerical}
While Theorem~\ref{thm.lambdaconvergence.discrete} is a general existence result, its proof does suggest how we can actually find a relaxed equilibrium: as the original MDP \eqref{J^pi} can be approximated by a sequence of entropy-regularized ones (indexed by $\lambda>0$), one can compute a relaxed equilibrium for each regularized problem and the limit (as $\lambda\to 0$) will be a relaxed equilibrium of the original problem. 

This method is numerically viable as it circumvents the set-valued fixed-point operator associated with \eqref{J^pi}, i.e., $\Phi$ in \eqref{Phi}. Indeed, it is difficult numerically to implement a set-valued fixed-point iteration, such that finding a relaxed equilibrium for \eqref{J^pi} directly is not easy at all. By contrast, as each regularized problem is associated with a single-valued fixed-point operator, i.e., $\Phi_\lambda$ in \eqref{Phi_lambda}, a fixed-point iteration can be implemented in a straightforward way. Certainly, to make this method fully rigorous, it remains to show the theoretic convergence of the single-valued fixed-point iteration, which is a nontrivial problem in itself and will be left for future research.
\end{Remark}
}

\section{Relaxed Equilibria in Continuous Time}\label{sec:continuous}
In this section, we take up the same setup in the first two paragraphs of Section \ref{sec:discrete}, except that the controlled process $X$ is now a continuous-time Markov chain. Specifically, each action $u\in U$ is associated with a $d\times d$ rate matrix (or, generator) $q^u$; namely, for each fixed $i\in\Sc$, $q^u_{ij}\ge 0$ for all $j\neq i$ and $q^u_{ii} = -\sum_{j\neq i} q^u_{ij}$. In addition, each $\mu\in\Pc(U)$ is associated with a $d\times d$ {\it relaxed} rate matrix $Q^\mu$, defined by $Q^\mu_{ij} := \int_U q^u_{ij} \mu(du)$ for all $i,j\in\Sc$. At any current state $i\in\Sc$, given a relaxed feedback control $\pi:\Sc\to\Pc(U)$, the dynamics of $X$ is dictated by $Q^{\pi(i)}_i$, the $i^{th}$ row of $Q^{\pi(i)}$. That is, the time until the next jump to other states is exponentially distributed with parameter $-Q^{\pi(i)}_{ii}$ and the jump will take $X$ to state $j\neq i$ with probability $-{Q^{\pi(i)}_{ij}}/{Q^{\pi(i)}_{ii}}$.

For any $\pi\in\Pi$, the corresponding value function is given by 
\begin{equation}\label{tJ^pi}
	\widetilde J^\pi(i) := \E_i\bigg[ \int_0^\infty f^{\pi(X^\pi_s)}\left(s,X^\pi_s\right) ds\bigg],\quad \forall i\in \Sc.
\end{equation}
Similarly to \eqref{eq.relaxed.J}, for any $\lambda> 0$ and $\pi\in\Pi_r$, we consider 
\be\label{eq.relaxed.tJ} 
\widetilde J^\pi_\lambda(i):= \E_i\bigg[ \int_0^\infty \left( f^{\pi(X_s)}(s,X^\pi_s) + \lambda  \delta(s)\Hc\big({\pi(X_s)}\big) \right) ds\bigg],\quad \forall i\in \Sc,
\ee 
To ensure that $\widetilde J^\pi$ and $\widetilde J^\pi_\lambda$ are well-defined, we impose the following.

\begin{Assumption}\label{assume.f.Q}
	$\widetilde M := \int_0^\infty \left(\sup_{i\in \Sc, u\in U}|f(s,i,u)|+ \delta(s) \right)ds <\infty$.
\end{Assumption}
We will commonly write $\tJ^\pi$ and $\tJ^\pi_\lambda$ as vectors in $\R^d$, i.e., 
$\tJ^\pi = (\tJ^\pi(1), \tJ^\pi(2),...,\tJ^\pi(d))$ and $\tJ^\pi_\lambda = (\tJ^\pi_\lambda(1), \tJ^\pi_\lambda(2),...,\tJ^\pi_\lambda(d))$.

Similarly to Definition~\ref{def.relaxedequi.discrete}, to formulate an equilibrium strategy in continuous time, we introduce, for any $\pi, \pi'\in \Pi$, the concatenation of $\pi'$ and $\pi$ at time $\eps>0$, denoted by $\pi'\otimes_\eps\pi\in\Pi$. Using this concatenated relaxed control means that the evolution of $X$ is governed first by $\pi'$ on the time interval $[0,\eps)$ and then by $\pi$ on $[\eps,\infty)$. 

\begin{Definition}
	Given $\lambda> 0$, we say $\pi\in\Pi_r$ is a {\color{black}regular relaxed equilibrium} (for \eqref{eq.relaxed.tJ}) if 
	\be\label{eq.def.weakrelax} 
	\limsup_{\eps\downarrow 0}\frac{\widetilde J^{\pi'\otimes_\eps \pi}_\lambda(i)- \widetilde J^{\pi}_\lambda(i)}{\eps}\leq 0,\quad \forall \pi'\in \Pi_r\ \hbox{and}\ i\in \Sc.
	\ee
	Similarly, we say $\pi\in\Pi$ is a relaxed equilibrium 
	(for \eqref{tJ^pi}) if it satisfies \eqref{eq.def.weakrelax} 
	with $\widetilde J^{\pi'\otimes_\eps \pi}_\lambda$, $\widetilde J^{\pi}_\lambda$, and $\pi'\in\Pi_r$ therein replaced by $\widetilde J^{\pi'\otimes_\eps \pi}$, $\widetilde J^{\pi}$, and $\pi'\in\Pi$, respectively. 
\end{Definition}

Similarly to \eqref{eq.relaxed.V}, for any $\pi\in\Pi$, we introduce the auxiliary value function 
\be\label{eq.relaxed.tV} 
\widetilde V^\pi(t,i):= \E_i \bigg[ \int_0^\infty f^{\pi(X_s)}(t+s,X_s) ds \bigg], \quad \forall (t,i)\in[0,\infty)\times\Sc. 
\ee 
In addition, similarly to \eqref{eq.relaxed.V.lambda}, for any $\lambda>0$ and $\pi\in\Pi_r$, we introduce the auxiliary value function 
\be\label{eq.relaxed.tV.lambda} 
\widetilde V^\pi_\lambda(t,i):= \E_i \left[ \int_0^\infty \left( f^{\pi(X_s)}(t+s,X_s) + \lambda\delta(t+s) \Hc(\pi(X_s)) \right)ds \right],\quad \forall (t,i)\in[0,\infty)\times\Sc.
\ee
For convenience, we will commonly write $\tV^\pi_\lambda(t)$ as a vector in $\R^d$, i.e., 
\[
\widetilde V^\pi_\lambda(t) = (\tV^\pi_\lambda(t,1), \tV^\pi_\lambda(t,2),...,\tV^\pi_\lambda(t,d))\in\R^d.
\]
We will write $\tV^\pi$ as a vector in $\R^d$ in the same manner.  

\subsection{Existence of {\color{black}Regular Relaxed Equilibria} for Entropy-Regularized MDP \eqref{eq.relaxed.tJ}}\label{subsec.cts.entropy}
\begin{Lemma}\label{lm.expansion}
	Let Assumption \ref{assume.f.Q} hold and $f(\cdot,i,u)$ and $\delta(\cdot)$ be continuous on $[0,\infty)$. Given $\lambda>0$, it holds for all $i\in\Sc$ and $\pi,\pi'\in\Pi_r$ that 
	\be\label{eq.lm.dis1} 
	\tJ^{\pi'\otimes_{\eps}\pi}_\lambda(i) = \tV^\pi_\lambda(\eps,i) + \left(f^{\pi'(i)}(0,i)+\lambda \Hc(\pi'(i))+Q^{\pi'(i)}_i\cdot \tV^\pi_\lambda(\eps) \right)\eps +o(\eps),\quad \hbox{as $\eps\downarrow 0$}. 
	\ee
	Similarly,  it holds for all $i\in\Sc$ and $\pi,\pi'\in\Pi$ that 
	\be\label{eq.lm.dis1'} 
	\tJ^{\pi'\otimes_{\eps}\pi}(i) = \tV^\pi(\eps,i) + \left(f^{\pi'(i)}(0,i)+Q^{\pi'(i)}_i\cdot \tV^\pi(\eps) \right)\eps +o(\eps),\quad \hbox{as $\eps\downarrow 0$}. 
	\ee
\end{Lemma}

\begin{proof}
	The result follows from a similar argument in \citet[Lemma 1]{HZ21}. 
\end{proof}

Based on Lemma~\ref{lm.expansion}, we can now generalize the fixed-point characterization in Proposition~\ref{propdiscr.entropy.cha} to continuous time. Consider a functional $\widetilde\Gamma_\lambda: \R^d\times\Sc\to \Dc(U)$ defined by
\begin{equation}\label{tGamma_lambda}
	u\mapsto\widetilde\Gamma_\lambda(y,i) (u) : = \frac{e^{\frac{1}{\lambda} \left(f(0,i,u) +q^u_i\cdot y\right)} }{\int_{U} e^{\frac{1}{\lambda} \left(f(0,i,v) +q^v_i\cdot y\right)} dv}\in \Dc(U),\quad \forall (y,i)\in \R^d\times\Sc.
\end{equation}
It follows that $\widetilde\Gamma_\lambda(y,\cdot)\in\Pi_r$ for all $y\in\R^d$. We can then define an operator $\widetilde\Psi_\lambda:\R^d\to\R^d$ by 
\[
\widetilde\Psi_\lambda(y) := 
\tJ^{\widetilde\Gamma_\lambda (y,\cdot)}_\lambda. 
\]
Moreover, for any $\pi\in\Pi_r$, as $\widetilde\Gamma_\lambda(\tJ^\pi_\lambda, \cdot)\in\Pi_r$,  
we can define an operator $\widetilde\Phi_\lambda:\Pi_r\to\Pi_r$ by 
\[
\widetilde\Phi_\lambda(\pi) := \widetilde\Gamma_\lambda\big(\tJ^\pi_\lambda,\cdot\big). 
\]

\begin{Proposition}\label{propconti.entropy.cha}
	Let Assumption \ref{assume.f.Q} hold and $f(\cdot,i,u)$ and $\delta(\cdot)$ be continuous on $[0,\infty)$.  For any $\lambda>0$, 
	\[
	\hbox{$\pi\in \Pi_r$ is a {\color{black}regular relaxed equilibrium}}\  \iff\  \widetilde\Phi_\lambda(\pi)=\pi.
	\]
\end{Proposition}
The proof of Proposition~\ref{propconti.entropy.cha} is relegated to Appendix~\ref{subsec:proof of propconti.entropy.cha}.

\begin{Corollary}\label{coro:Psi'}
	Let Assumption \ref{assume.f.Q} hold and suppose that $f(\cdot,i,u)$ and $\delta(\cdot)$ are continuous on $[0,\infty)$. Given $\lambda>0$, it holds for all $y\in\R^d$, 
	\[
	\hbox{$y= \tJ^\pi_\lambda$ for some {\color{black}regular relaxed equilibrium} $\pi\in\Pi_r$}\ \iff\ \widetilde\Psi_\lambda(y)=y. 
	\]
	In particular, $\widetilde\Psi_\lambda(y)=y$ implies that $\widetilde\Gamma_\lambda(y,\cdot)\in\Pi_r$ is a {\color{black}regular relaxed equilibrium} for \eqref{eq.relaxed.tJ}. 
\end{Corollary}

\begin{proof}
	If $y=\tJ^\pi_\lambda$ for a {\color{black}regular relaxed equilibrium} $\pi\in\Pi_r$, then by Proposition~\ref{propconti.entropy.cha}, $\pi=\widetilde\Phi_\lambda(\pi)= \widetilde\Gamma_\lambda(\tJ^{\pi}_\lambda,\cdot)=\widetilde\Gamma_\lambda(y,\cdot)$, which implies $y=\tJ^\pi_\lambda=\tJ^{\widetilde\Gamma_\lambda(y,\cdot)}_\lambda=\widetilde\Psi_\lambda(y)$. Conversely, if $y= \widetilde\Psi_\lambda(y)=\tJ^{\widetilde\Gamma_\lambda(y,\cdot)}_\lambda$, set $\pi:= \widetilde\Gamma_\lambda(y,\cdot)\in\Pi_r$. Then, we have $y=\tJ^\pi_\lambda$ and thus $\pi = \widetilde\Gamma_\lambda(y,\cdot)=\widetilde\Gamma_\lambda(\tJ^{\pi}_\lambda,\cdot)=\widetilde\Phi_\lambda(\pi)$. By Proposition~\ref{propconti.entropy.cha}, this implies that $\pi$ is a {\color{black}regular relaxed equilibrium}.
\end{proof}

By a similar argument in the proof of Theorem \ref{thm.discrete.relaxedentropy} and using Corollary~\ref{coro:Psi'}, we can establish the existence of {\color{black}regular relaxed equilibria} for \eqref{eq.relaxed.tJ}. 

\begin{Theorem}\label{thm1}
	Let Assumptions \ref{assume.f.Q}, \ref{assume.u.lips} (with $p^u_i$ therein replaced by $q^u_i$), and \ref{assume.U.cone} hold. Suppose that $f(\cdot,i,u)$ and $\delta(\cdot)$ are continuous on $[0,\infty)$. For any $\lambda>0$, there exists $y\in\R^d$ such that $\widetilde\Psi_\lambda(y)=y$. Hence, $\widetilde\Gamma_\lambda(y,\cdot)\in\Pi_r$ is a {\color{black}regular relaxed equilibrium} for \eqref{eq.relaxed.tJ}. 
\end{Theorem}

\subsection{Existence of Relaxed Equilibrium for MDP \eqref{tJ^pi}}\label{subsec.cts.noentropy} 
Let us now move on to prove the existence of a relaxed equilibrium for \eqref{tJ^pi}. 
Suppose that $u\mapsto f(t,i,u)$ and $u\mapsto q^u_i$ are continuous. As $U$ is compact,  for any $(y,i)\in\R^d\times\Sc$,
\be\label{eq.conti.supp}  
\widetilde E(y,i) := \argmax_{u\in U} \{f(0,i, u)+ q^u_i\cdot y \}\subseteq U
\ee
is nonempty and closed. For any $y\in\R^d$, consider the collection 
\begin{equation}\label{tGamma}
	\widetilde\Gamma(y) := \left\{\pi\in\Pi : \operatorname{supp}(\pi(i))\subseteq E(y,i),\ \forall i\in \Sc \right\},
\end{equation}
where $\operatorname{supp}(\rho)$ denotes the support of $\rho\in\Dc(U)$. We can then define a set-valued operator $\widetilde \Psi:\R^d\to2^{\R^d}$ by 
\begin{equation}\label{tPsi}
	\widetilde\Psi(y) := \big\{\tJ^{\pi}:\pi\in\widetilde\Gamma(y)\big\}\subseteq \R^d.  
\end{equation}
Moreover, we can also define a set-valued operator $\widetilde\Phi:\Pi\to2^\Pi$ by 
\[
\widetilde\Phi(\pi) :=\widetilde\Gamma(\tJ^\pi)\subseteq\Pi. 
\]
We can then generalize the fixed-point characterizations in Proposition \ref{prop.chadiscrete.nonentropy} to continuous time.

\begin{Proposition}\label{prop.nonentropy.cha}
	Let Assumption \ref{assume.f.Q} hold and the maps $u\mapsto f(t,i,u)$ and $u\mapsto q^u_i$ be continuous. Then, 
	\begin{equation}\label{pi in tPhi}
		\hbox{$\pi\in \Pi$ is a relaxed equilibrium}\  \iff\  \pi\in \widetilde\Phi(\pi).
	\end{equation}
	Moreover, for any $y\in\R^d$, 
	\[
	\hbox{$y= \tJ^\pi$ for some relaxed equilibrium $\pi\in\Pi$}\ \iff\ y\in\widetilde\Psi(y). 
	\]
\end{Proposition}

\begin{proof}
	By the same argument in the proof of Proposition \ref{propconti.entropy.cha} (except that now we use \eqref{eq.lm.dis1'} instead of \eqref{eq.lm.dis1}), we observe that $\pi\in\Pi$ is a relaxed equilibrium if and only if $\pi(i)\in\Pc(U)$ fulfills 
	$$
	\pi(i)\in \argmax_{\mu\in \Pc(U)} \int_U \left( f(0,i,u) + q^u_i\cdot \tJ^\pi \right) \mu(du),\quad \forall i\in\Sc,
	$$ 
	which is equivalent to $\operatorname{supp}(\pi(i))\subseteq \widetilde E(\tJ^\pi,i)$ for all $i\in\Sc$, i.e., $\pi\in\widetilde\Gamma(\tJ^\pi)=\widetilde\Phi(\pi)$. 
	
	For any $y\in \R^d$ such that $y\in \widetilde \Psi(y)$, in view of \eqref{tPsi}, $y=\tJ^\pi$ for some $\pi\in \widetilde\Gamma(y)$. It follows that $\pi \in \widetilde\Gamma(y) = \widetilde\Gamma(\tJ^\pi)=\widetilde\Phi(\pi)$. By \eqref{pi in Phi}, this implies that $\pi$ is a relaxed equilibrium. Conversely, suppose that $y=\tJ^\pi$ for some relaxed equilibrium $\pi\in\Pi$. By \eqref{pi in Phi}, $\pi\in \widetilde\Phi(\pi) = \widetilde\Gamma(\tJ^\pi) = \widetilde\Gamma(y)$. With $y=\tJ^\pi$ and $\pi\in\widetilde\Gamma(y)$, we immediately conclude $y\in\widetilde\Psi(y)$.
\end{proof}

Given an arbitrary sequence $\lambda_n\to 0+$ with $v^n=\widetilde{\Psi}_{\lambda_n}(v^n)$ for all $n\in \N$, i.e., $\widetilde{\Gamma}_{\lambda_n}(v)$ is a relaxed equilibrium under $\lambda_n$, the next theorem follows from similar arguments in Theorem \ref{thm.lambdaconvergence.discrete}.

\begin{Theorem}\label{thm.lambdaconvergence.continuous}
	Let Assumptions \ref{assume.f.Q}, \ref{assume.u.lips} (with $p^u_i$ therein replaced by $q^u_i$), and \ref{assume.U.cone} hold. Then, a relaxed equilibrium for \eqref{tJ^pi} exists.
\end{Theorem}




\section{Discussion: When Do We Need Relaxed Equilibria?}\label{subsec:relaxed needed}
\subsection{The Discrete-time Case}\label{subsec:discrete}
Let us consider the discrete-time setup in Section~\ref{sec:discrete}. We first draw a detailed comparison between Theorem~\ref{thm.lambdaconvergence.discrete} and \citet[Theorem 4]{HZ21}. 

{\color{black}
\begin{Remark}\label{rem:more general p^u}
\citet{HZ21} consider the special case where $d=\ell$ and
\be\label{eq.delta} 
p^u_i := u_i\quad \forall i\in \Sc,\quad \text{with}\ u_i\in U\subseteq \{ y\in \R^d : y\ge 0,\ y_1+y_2+...+y_d=1\}.
\ee 
This essentially states that one can directly ``decide'' (rather than just ``influence'') the transition probabilities. By contrast, Theorem~\ref{thm.lambdaconvergence.discrete} only requires the map $u\mapsto p^u_i$ to be Lipschitz (Assumption~\ref{assume.u.lips}), without specifying any specific form of it. That is, Theorem~\ref{thm.lambdaconvergence.discrete} allows for more general (and potentially more realistic) dependence of transition probabilities on the control action $u\in U$. 
\end{Remark}
}

\begin{Remark}\label{rem:relaxed needed}
	By \cite[Theorem 4]{HZ21}, a standard equilibrium for \eqref{J^pi} exists, provided that {\color{black}(i) \eqref{eq.delta} holds, (ii) $U$ is convex, and (iii) $f(0,i,\cdot)$ is concave for all $i\in\Sc$. If one of the conditions fails, it is unclear whether standard equilibria exist. 
	
	In particular, when (iii) fails, \cite[Remark 10]{HZ21} shows that standard equilibria do not exist in a concrete example, where $\Sc=\{1,2\}$, $\delta(\cdot)$ is a quasi-hyperbolic discount function, $f(t,i,u)$ is of the form \eqref{discount problem} with $g(i,u)$  bounded and continuous (but non-concave) in $u$, and 
$p^{u}_i := u_i$ $\forall i\in \{1,2\}$ with $u_i\in U= \{ y\in \R^2: y\geq 0, y_1+y_2=1\}$.
This simple setup allows us to verify Assumptions \ref{assume.f.p}, \ref{assume.u.lips}, and \ref{assume.U.cone} immediately. Hence, although there is no standard equilibrium, a relaxed equilibrium exists in this example by Theorem~\ref{thm.lambdaconvergence.discrete}. }
\end{Remark}

{\color{black}
Remarks~\ref{rem:more general p^u} and \ref{rem:relaxed needed} show the importance of Theorem~\ref{thm.lambdaconvergence.discrete}: unlike standard equilibria, a relaxed equilibrium for \eqref{J^pi} exists much more generally, without the need of conditions (i), (ii), and (iii) in Remark \ref{rem:relaxed needed}. 
	Hence, in the face of general controlled transition probabilities (beyond the special form \eqref{eq.delta}) and non-concave reward functions, relaxed equilibria should certainly be considered.  
}


{\color{black}
After learning that relaxed equilibria are needed in cases where standard equilibria may not exist, let us now investigate the other side of the picture: for cases where both standard and relaxed equilibria are known to exist, does the consideration of relaxed equilibria provide any added value?
As shown in the next result, 
it is {\it not} necessary to consider relaxed equilibria in such a case, as any value achieved by a relaxed equilibrium can be recovered by a suitable standard equilibrium. 
}


\begin{Proposition}\label{prop.convex.nonentropy'} 
	Assume \eqref{eq.delta} and that $U$ therein is convex. Suppose that $f(t,i,u)$ takes the form \eqref{discount problem}, with $g(i, \cdot)$ therein concave for all $i\in \Sc$. Also, let Assumptions \ref{assume.f.p}, \ref{assume.u.lips}, and \ref{assume.U.cone} hold.  Then, 
	for any relaxed equilibrium $\pi\in\Pi$ for \eqref{J^pi}, $\alpha^\pi:\Sc\to \R^d$ defined by $\alpha^\pi (i):= \int_U u (\pi(i))(du)$, $\forall i\in\Sc$, is a standard equilibrium such that $J^{\alpha^\pi}(\cdot)=J^\pi(\cdot)$. 
\end{Proposition}

\begin{proof}
	Note that the convexity of $U$ ensures that $\alpha^\pi (i)\in U$ for all $i\in\Sc$. As $\pi\in\Pi$ is a relaxed equilibrium for \eqref{J^pi}, by \eqref{pi in Phi}, 
	\be\label{E special}
	\operatorname{supp}(\pi(i))\subseteq E(V^\pi,i) = \argmax_{u\in {U}} \{ g(i,u) + u \cdot V^{\pi} \},\quad\forall i\in\Sc, 
	\ee
	where the equality follows from \eqref{eq.discrete.supp} and \eqref{eq.delta}. Given $i\in\Sc$, if $E(V^\pi,i)$ is a singleton, \eqref{E special} implies that $\pi(i)\in\Pc(U)$ concentrates on the unique element in $E(V^\pi,i)$, which coincides with $\alpha^\pi$ by definition. Hence, $f^{\pi(i)}(t,i) = f^{\alpha^\pi}(t,i)$ trivially holds for all $t\in\N_0$. If $E(V^\pi,i)$ is not a singleton, in view of its form in \eqref{E special}, $u\mapsto g(i,u)$ must be linear on $E(V^\pi,i)$. It follows that
	\[
	f^{\pi(i)}(t,i) = \int_U \delta(t)g(i,u) (\pi(i))(du) = \delta(t) g(i, \alpha^\pi) = f^{\alpha^\pi}(t,i),\quad\forall t\in\N_0. 
	\]
	With $f^{\pi(i)}(t,i) = f^{\alpha^\pi}(t,i)$ for all $t\in\N_0$ and $i\in\Sc$, we conclude from \eqref{J^pi} that $J^{\alpha^\pi}(\cdot)=J^\pi(\cdot)$.  
\end{proof}

{\color{black}
It is interesting to note that \citet{JN21} also establish results in the same spirit as Proposition~\ref{prop.convex.nonentropy'}, under a specific form of discounting. 

\begin{Remark}
	Under quasi-hyperbolic discounting, \cite[Theorem 3.4]{JN21} asserts the existence of a relaxed equilibrium $\pi$ that is almost a standard one: for each state $i$, the support of the probability $\pi(i)$ contains at most two points in the action space. Moreover, \cite[Theorem 3.5]{JN21} shows that additional ``atomless'' assumptions can further reduce the support of $\pi(i)$ to contain only one point, i.e., the relaxed equilibrium reduces to a standard one. As opposed to Proposition~\ref{prop.convex.nonentropy'}, these results hold for fairly general reward functions (without the need of concavity) and transition probabilities (without the need of linear dependence on $u$).
		
		Such generality, however, hinges crucially on quasi-hyperbolic discounting. In a nutshell, by using the structure of quasi-hyperbolic discounting, the original time-inconsistent problem can be expressed as a functional of an auxiliary time-consistent problem (under exponential discounting); see \cite[(2.3)-(2.4)]{JN21}. A detailed analysis is then performed on the time-consistent problem, which contributes to proving \cite[Theorems 3.4 and 3.5]{JN21}; see \cite[Section 7]{JN21}. This method, quite specific to quasi-hyperbolic discounting, cannot be easily extended to the case of a general discount function. 
\end{Remark}
}

Finally, we would like to point out that the need of relaxed controls 
is unique to the case of time inconsistency. When the problem $\sup_{\pi\in\Pi} J^\pi(i)$ is time-consistent (e.g., under exponential discounting), one can consider without loss of generality only standard controls. 

\begin{Remark}\label{rm.discrete}
	Let $f(t,i,u)$ takes the form \eqref{discount problem} with $\delta(t) = e^{-\beta t}$ for some $\beta>0$. As there is no time inconsistency under exponential discounting, the optimal value $J^*(i):=\sup_{\pi\in \Pi} J^\pi(i)$ fulfills the Bellman equation
	\be\label{eq.rm} 
	-J^*(i)+\sup_{\mu\in \Pc(U)}\int_U \left(g(i,u)+ e^{-\beta} p^u_i\cdot J^*\right)\mu(du)=0,\quad \forall i\in\Sc.
	\ee 
	Suppose that an optimal relaxed control $\pi^*\in\Pi$ exists, i.e., $J^{\pi^*} = J^*$. Our goal is to show that $J ^*$ can be achieved by a standard control---such that there is no need to consider relaxed controls. 
	
	Given $i\in\Sc$, let $u\mapsto p^u_i$ and $u\mapsto g(i,u)$ be continuous such that $\argmax_{u\in U}\{ g(i,u) + e^{-\beta} p^u_i\cdot J^* \}$ is nonempty. For any standard control $\alpha$ with $\alpha(i)\in \argmax_{u\in U}\{ g(i,u) + e^{-\beta} p^u_i\cdot J^* \}$ for all $i\in\Sc$,
	\[
	J^{\alpha\otimes_1\pi^*}(i)= g(i,\alpha(i)) + e^{-\beta} p^{\alpha(i)}_i\cdot J^* =  J^*(i),\quad \forall i\in\Sc,
	\]
	where the second equality follows from \eqref{eq.rm} and $\alpha(i)\in \argmax_{u\in U}\{ g(i,u) + e^{-\beta} p^u_i\cdot J^* \}$. Similarly, 
	\[
	J^{\alpha\otimes_2\pi^*}(i)= g(i,\alpha(i)) + e^{-\beta} p^{\alpha(i)}_i\cdot J^{\alpha\otimes_1\pi^*} =  g(i,\alpha(i)) + e^{-\beta} p^{\alpha(i)}_i\cdot J^{\pi^*} =  J^*(i),\quad \forall i\in\Sc.
	\]
	Iterating this argument yields $J^{\alpha\otimes_m \pi^*} = J^*$ for all $m\in\N$. It follows that
	\begin{align*}
		J^*(i) = \lim_{m\to\infty} J^{\alpha\otimes_m \pi^*}(i) = \lim_{m\to\infty} \left(J^\alpha(i) + e^{-\beta m}\E\big[p^{\alpha(X_{m-1})}_{X_{m-1}}\cdot (J^* - J^\alpha)\big]\right) = J^\alpha(i),\quad \forall i\in\Sc. 
	\end{align*}
	That is, the optimal value is achieved by the standard feedback control $\alpha$.
\end{Remark}

\subsection{The Continuous-time Case}\label{subsec:continuous} 
Let us consider the continuous-time setup in Section~\ref{sec:continuous}. As we will see, many observations we made in Section~\ref{subsec:discrete} still hold in continuous time. To begin with, we draw a detailed comparison between Theorem~\ref{thm.lambdaconvergence.continuous} and \citet[Theorem 3]{HZ21}. 

{\color{black}
\begin{Remark}\label{rem:more general q^u}
In \cite{HZ21}, one considers the special case where $\ell=d-1$ and
\be\label{eq.delta'} 
q^u_i := \bigg(u_1,...,u_{i-1}, -\sum_{j\neq i}u_j, u_{i+1}, ...,u_d\bigg)\quad \forall i\in \Sc,\quad \text{with}\ u\in U\subseteq \{ y\in \R^{d-1} : y\ge 0\}.
\ee 
This essentially states that one can directly ``decide'' (rather than just ``influence'') the transition rates. By contrast, Theorem~\ref{thm.lambdaconvergence.continuous} only requires the map $u\mapsto q^u_i$ to be Lipschitz (Assumption~\ref{assume.u.lips}, with $p^u_i$ therein replaced by $q^u_i$), without specifying any specific form of it. That is, Theorem~\ref{thm.lambdaconvergence.continuous} allows for more general (and potentially more realistic) dependence of transition rates on the control action $u\in U$. 
\end{Remark}
}

{\color{black}
By \cite[Theorem 3]{HZ21}, a standard equilibrium for \eqref{tJ^pi} exists, provided that (i) \eqref{eq.delta'} holds, (ii) $U$ is convex, and (iii) $f(0,i,\cdot)$ is concave for all $i\in\Sc$. If one of the conditions fails, it is unclear whether standard equilibria exist. In particular, when (iii) fails, the next example has no standard equilibrium. 
}

\begin{Example}\label{eq:no standard E}
	Let $\Sc=\{1,2\}$ and $U=[0,1]$. For any $u\in U$, we take the rate matrix $q^u$ to be $q^u_1=(-u,u)$ and $q^u_2=(u,-u)$. Consider $\delta(t):=(e^{-t}+e^{-2t})/2$ for $t\ge 0$ and take $f(t,i,u)=\delta(t)g_i(u)$, $i=1,2$, where $g_1(u) :=-\frac{7}{8}\sqrt{u}$ and $g_2(u):=\frac{19}{9}-\sqrt{1-u}$ are strictly convex on $U$.  {\color{black} Thanks to \cite[Theorem 1]{HZ21},  
	$\alpha=(\alpha(1),\alpha(2))= (a^*,b^*)$ is a standard equilibrium} if and only if  
\begin{align*}
	&\argmax_{a\in [0,1]} \left\{ g_1(a)-a(V^{(a^*,b^*)}(0,1)- V^{(a^*,b^*)}(0,2))\right \}=a^*,\\
	&\argmax_{b\in [0,1]} \left\{ g_2(b)+b(V^{(a^*,b^*)}(0,1)- V^{(a^*,b^*)}(0,2)) \right\}=b^*. 
\end{align*}
The strict convexity of $g_1$ and $g_2$ then implies that there are only four possibilities for a standard equilibrium $\alpha=(a^*,b^*)$, i.e., $(0,0)$, $(0,1)$, $(1,0)$, and $(1,1)$. 
{\color{black}By calculations similar to \cite[(37)-(39)]{HZ21}, we have }
\begin{align*}
	\begin{cases}
		V^{(0,0)}(0,1)-V^{(0,0)}(0,2) &=-\frac{3}{4}\cdot \frac{10}{9}=-\frac{5}{6}; \\ 
		V^{(1,0)}(0,1)-V^{(1,0)}(0,2)&=-\frac{5}{12}\cdot (\frac{15}{8}+\frac{1}{9})\in (-\frac{7}{8},0); \\
		V^{(0,1)}(0,1)-V^{(0,1)}(0,2)&=-\frac{5}{12}\cdot (2+\frac{1}{9})<-\frac{7}{8}; \\
		V^{(1,1)}(0,1)-V^{(1,1)}(0,2)&=-\frac{7}{24}\cdot (\frac{23}{8}+\frac{1}{9})\in (-\frac{7}{8}, 0),
	\end{cases}
\end{align*}
which imply
$$
\begin{cases}
	\argmax_{b\in [0,1]} \{ g_2(b)+b(V^{(0,0)}(0,1)- V^{(0,0)}(0,2)) \}=1\neq 0;\\
	\argmax_{a\in [0,1]} \{ g_1(a)-a(V^{(1,0)}(0,1)- V^{(1,0)}(0,2)) \}=0\neq 1;\\
	\argmax_{a\in [0,1]} \{ g_1(a)-a(V^{(0,1)}(0,1)- V^{(0,1)}(0,2)) \}=1\neq 0;\\
	\argmax_{a\in [0,1]} \{ g_1(a)-a(V^{(1,1)}(0,1)- V^{(1,1)}(0,2)) \}=0\neq 1.
\end{cases}
$$
Therefore, we conclude that there exists no standard equilibrium. {\color{black} Despite this, since Assumptions \ref{assume.f.Q}, \ref{assume.u.lips} (with $p^u_i$ therein replaced by $q^u_i$), and \ref{assume.U.cone} can be immediately verified in the present setting, Theorem~\ref{thm.lambdaconvergence.continuous} asserts that a relaxed equilibrium exists. }
\end{Example}

{\color{black}
Remark~\ref{rem:more general q^u} and Example~\ref{eq:no standard E} show the importance of Theorem~\ref{thm.lambdaconvergence.continuous}: unlike standard equilibria, a relaxed equilibrium for \eqref{J^pi} exists much more generally, without the need of conditions (i), (ii), and (iii) mentioned above Example~\ref{eq:no standard E}. 
	Hence, in the face of general controlled transition rates (beyond the special form \eqref{eq.delta'}) and non-concave reward functions, relaxed equilibria should certainly be considered.  
}


{\color{black}
Let us now investigate the other side of the picture: for cases where both standard and relaxed equilibria are known to exist, does the consideration of relaxed equilibria provide any added value?
As shown in the next result, 
it is {\it not} necessary to consider relaxed equilibria in such a case, as any value achieved by a relaxed equilibrium can be recovered by a suitable standard equilibrium. 
}



\begin{Proposition}\label{prop.convex.nonentropy} 
Assume \eqref{eq.delta'} and that $U$ therein is convex. Suppose that $f(t,i,u)$ takes the form \eqref{discount problem}, with $g(i, \cdot)$ therein concave for all $i\in \Sc$. Also, let Assumptions \ref{assume.f.p}, \ref{assume.u.lips} (with $p^u_i$ therein replaced by $q^u_i$), and \ref{assume.U.cone} hold.  Then, 
for any relaxed equilibrium $\pi\in\Pi$ for \eqref{J^pi}, $\alpha^\pi:\Sc\to \R^d$ defined by $\alpha^\pi (i):= \int_U u (\pi(i))(du)$, $\forall i\in\Sc$, is a standard equilibrium such that $J^{\alpha^\pi}(\cdot)=J^\pi(\cdot)$. 
\end{Proposition}

The proof of Proposition~\ref{prop.convex.nonentropy} is similar to that of Proposition \ref{prop.convex.nonentropy'}, with $p^u_i$ and $V^\pi(\cdot)$ therein replaced by $q^u_i$ and $\tV^\pi(0,\cdot)$, respectively.

Finally, we would like to point out that, in line with the discrete-time setting, the need of relaxed controls 
is unique to the case of time inconsistency. When the problem $\sup_{\pi\in\Pi} \tJ^\pi(i)$ is time-consistent (e.g., under exponential discounting), one only needs to consider standard controls. 

\begin{Remark}\label{rm.conti}
Let $f(t,i,u)$ take the form \eqref{discount problem} with $\delta(t) = e^{-\beta t}$ for some $\beta>0$. As there is no time inconsistency under exponential discounting, the optimal value $\tJ^*(i):=\sup_{\pi\in \Pi} \tJ^\pi(i)$ fulfills the Bellman equation
$$
-\beta \tJ^*(i)+\sup_{\mu \in \Pc(U)}\int_U \left(g(i,u) + q^u_i\cdot \tJ^*\right)\mu(du)=0.
$$
Suppose that an optimal relaxed control $\pi^*\in\Pi$ exists, i.e., $\tJ^{\pi^*} = \tJ^*$. Then, arguments in Remark~\ref{rm.discrete} can be adapted to the present continuous-time setting and show that $\tJ^{\alpha} = \tJ^{\pi^*}$ for any standard control $\alpha$ with $\alpha(i)\in\argmax_{u\in U}\{g(i,u) + q^u_i\cdot \tJ^*\}$ for all $i\in\Sc$. 
\end{Remark}

\section{Convergence from Discrete Time to Continuous Time}\label{sec:disctsconverges}
In this section, we take up the same continuous-time setup as in Section~\ref{sec:continuous}. For a step size $h>0$ small enough, we construct a discrete-time approximation as follows. Let $\{X^h_k\}_{k\in\N_0}$ be a discrete-time controlled Markov process as in Section~\ref{sec:discrete}, with its transition matrix given by
\be\label{eq.phtoQ} 
(p_h)^u_i := h q^u_i+e_i\quad \forall u\in U,\ i\in\Sc,
\ee 
where $\{e_i\}_{i=1}^d$ is the standard basis of $\R^d$. 
Consider the reward function 
\be\label{f_h}
f_h(k,i,u):= f(kh, i,u) h\quad \forall k\in \N_0, 
\ee
where $f(t,i,u)$ is the reward function from Section~\ref{sec:continuous}. For any $\pi\in\Pi$, $J^\pi$ in \eqref{J^pi} now takes the form
\begin{equation}\label{J^pi,h}
J^{h,\pi}(i) := \E_i\bigg[ \sum_{k=0}^\infty f^{\pi(X^h_k)}_h\left(k,X^h_k\right)\bigg]= h \E_i\bigg[ \sum_{k=0}^\infty \int_{U} f(kh, X_{k}^{h}, u) (\pi(X_k^h)) (du) \bigg],\quad \forall i\in \Sc.
\end{equation}
Similarly, the auxiliary value function $V^{h,\pi}(i)$ is defined as in \eqref{eq.relaxed.V} with $f$ and $X^\pi$ therein replaced by $f_h$ and $X^h$, respectively. 
In addition, for any $\lambda>0$ and $\pi\in\Pi_r$, we recall $J^\pi_{\lambda}$ in \eqref{eq.relaxed.J} and define
\begin{align}\label{eq.relaxed.J^h} 
J^{h,\pi}_\lambda(i) := J^\pi_{h \lambda}(i) &= \E_i\bigg[ \sum_{k=0}^\infty \left( f^{\pi(X^h_k)}_h(k,X^h_k) + h\lambda  \delta(k)\Hc\big({\pi(X^h_k)}\big) \right) \bigg]\notag\\
& = h \E_i\bigg[ \sum_{k=0}^\infty \left( \int_{U} f(kh, X_{k}^{h}, u) (\pi(X_k^h)) (du) + \lambda  \delta(k)\Hc\big({\pi(X^h_k)}\big) \right) \bigg],\quad \forall i\in \Sc.
\end{align} 
Similarly, the auxiliary value function $V^{h,\pi}_\lambda(i)$ is defined as in \eqref{eq.relaxed.V.lambda} with $f$, $X$, and $\lambda$ replaced by $f_h$, $X^h$, and $h\lambda$, respectively. 



Let us now we study the convergence of value functions from discrete time to continuous time. 

\begin{Lemma}\label{lm.sec3.1}
Let Assumptions \ref{assume.f.Q}, \ref{assume.u.lips} (with $p^u_i$ therein replaced by $q^u_i$), and \ref{assume.U.cone} hold.  Suppose that $f(\cdot,i,u)$ and $\delta(\cdot)$ are continuous on $[0,\infty)$.
Take any sequence $\{h_n\}_{n\in \N}$ in $(0,1]$ with $h_n\downarrow 0$. 
\bi  
\item[(a)] For any $\{\pi^n\}_{n\in\N}$ and $\pi^\infty$ in ${\color{black} \Pi}$ such that $\pi^n(i)\to \pi^\infty(i)$ weakly for all $i\in \Sc$,  
$$
V^{h_n,\pi^n}(i)\to \tV^{\pi^\infty}(0,i) = \tJ^{\pi^\infty}(i),\quad i\in\Sc.
$$
\item[(b)] For any $\lambda>0$ and $\{y^n\}_{n\in\N}$ with $y^n\to y^\infty$ for some $y^\infty\in\R^d$, we have $ \Gamma_{h_n\lambda} (y^n,i) \to \Gamma_\lambda (y^\infty,i)$ for all $i\in \Sc$. Therefore,
$$
V^{h_n,\Gamma_{h_n\lambda}(y^n,\cdot)}_{\lambda}(i)\to \tV^{\widetilde\Gamma_{\lambda}(y^\infty,\cdot)}(0,i)=\tJ^{\widetilde\Gamma_{\lambda}(y^\infty,\cdot)}(i),\quad \forall i\in\Sc.
$$
\ei 
\end{Lemma}
The proof of Lemma~\ref{lm.sec3.1} is relegated to Appendix~\ref{subsec:proof of lm.sec3.1}. 

We are ready to present our main convergence result from discrete time to continuous time: as the time step $h_n>0$ tends to zero, a {\color{black}regular relaxed equilibrium (resp. a relaxed equilibrium)} $\pi^n$  in discrete time (with time step $h_n$) ultimately converges to a {\color{black}regular relaxed equilibrium (resp. a relaxed equilibrium)} in continuous time.  

\begin{Theorem}\label{thm.discts.convergence}
Let Assumptions \ref{assume.f.Q}, \ref{assume.u.lips} (with $p^u_i$ therein replaced by $q^u_i$), and \ref{assume.U.cone} hold. Suppose that $f(\cdot,i,u)$ and $\delta(\cdot)$ are continuous on $[0,\infty)$. Take any $\{h_n\}_{n\in \N}$ in $(0,1]$ with $h_n\downarrow 0$. 
\bi  
\item[(a)] Given $\lambda>0$, 
let $\pi^n\in\Pi_r$ be a {\color{black}regular relaxed equilibrium} for $J^{h_n,\pi}_\lambda$ in \eqref{eq.relaxed.J^h} for all $n\in\N$. Then, for each $i\in\Sc$, $\pi^n(i)\in\Dc(U)$ converges pointwise to some $\pi^\infty(i)\in\Dc(U)$, up to a subsequence. Moreover, $\pi^\infty\in\Pi_r$ is a {\color{black}regular relaxed equilibrium} for $\tJ^\pi_\lambda$ in \eqref{eq.relaxed.tJ}. 

\item [(b)] Let $\pi^n\in\Pi$ be a relaxed equilibrium for $J^{h_n,\pi}$ in \eqref{J^pi,h} for all $n\in\N$. Then, for each $i\in\Sc$, $\pi^n(i)\in\Pc(U)$ converges weakly to some $\pi^\infty(i)\in\Pc(U)$, up to a subsequence. Moreover, $\pi^\infty\in\Pi$ is a relaxed equilibrium for $\tJ^\pi$ in \eqref{tJ^pi}.
\ei 
\end{Theorem}
The proof of Theorem~\ref{thm.discts.convergence} is relegated to Appendix~\ref{subsec:proof of thm.discts.convergence}.



\appendix
\section{Proofs}
\subsection{Proof of Lemma~\ref{lm.1}}\label{subsec:proof of lm.1}
	Fix $i\in\Sc$ and $y\in\R^d$. For an arbitrary $\bar u\in U$, thanks to Assumption \ref{assume.u.lips},	
	\begin{equation}\label{Lipest}
		\left|(f(0,i,u)+p^u_i\cdot y)- (f(0,i,\bar{u})+p^{\bar u}_i\cdot y)\right|\leq \Theta(1 + |y|)|u-\bar{u}|, \quad \forall u\in {U}.
	\end{equation}
	This, along with \eqref{Gamma_lambda}, implies
	\be\label{eq.lm.ent1} 
	\begin{aligned}
		\Gamma_\lambda(y,i)(\bar u)= \frac{\exp(\frac{1}{\lambda}[f(0,i,\bar{u})+p^{\bar u}_i\cdot y] )}{\int_{U} \exp(\frac{1}{\lambda}[f(0,i,u)+p^u_i\cdot y] )du}
		\leq \frac{1}{\int_{U} \exp \big(-\frac{\Theta}{\lambda}[ (1+ |y|)|u-\bar{u}| ] \big) du},\quad \forall u\in {U}, 
	\end{aligned}
	\ee
	where the inequality follows from dividing by $\exp(\frac1\lambda [f(0,i,\bar{u})+p^{\bar u}_i\cdot y])$ the numerator and the denominator and then using the estimate \eqref{Lipest}. 
	
	{\color{black}Now let us prove \eqref{eq:supH esti} for $\ell>1$.} By Assumption \ref{assume.U.cone}, 
	\begin{align}\label{eq.lm.ent} 
		\int_U e^{-\frac{\Theta}{\lambda} (1 + |y|)|u-\bar u|}  du &\ge \int_{\operatorname{cone}(\bar u,\iota)\cap B_\vartheta(\bar u)} e^{-\frac{\Theta}{\lambda} (1 + |y|)|u-\bar u|}  du = \int_{\Delta_\iota \cap B_\vartheta(0)} e^{-\frac{\Theta}{\lambda} (1 + |y|) |u|} du,
	\end{align}
	where the equality follows from a translation from $\bar u$ to $0\in\R^\ell$ and an appropriate rotation about $0\in\R^\ell$. Let us estimate the right-hand side of \eqref{eq.lm.ent} in the next two cases. If $\frac{\Theta}{\lambda} (1+ |y|)\vartheta \leq 1$, 
	$$
	\int_{\Delta_\iota \cap B_\vartheta(0)} e^{-\frac{\Theta}{\lambda} (1 + |y|) |u|} du \geq e^{-1} \Leb(\Delta_\iota \cap B_\vartheta(0)) =: K_0. 
	$$
	If $\frac{\Theta}{\lambda} (1+ |y|)\vartheta > 1$, consider the two positive constants
	\[
	K_1 :=  \int_0^\pi\sin^{\ell-2}(\varphi_1)d\varphi_1 \cdots \int_0^\pi \sin (\varphi_{\ell-2})d\varphi_{\ell-2} \int_{-\iota}^{\iota} d \varphi_{\ell-1}\quad \hbox{and}\quad K_2 := \int_0^1z^{\ell-1}e^{-z}dz. 
	\]
	By using the $\ell$-dimensional spherical coordinates, we have
	\bee
	\begin{aligned}
		\int_{\Delta_\iota \cap B_\vartheta(0)} e^{-\frac{\Theta}{\lambda} (1 + |y|) |u|} du &= K_1 \int_0^\vartheta r^{\ell-1}  e^{-\frac{\Theta}{\lambda} (1 + |y|) r} dr\\
		& = K_1 \bigg(\frac{\lambda}{\Theta (1+ |y| )}\bigg)^\ell \int_0^{\frac{\Theta}{\lambda}(1 + |y|)\vartheta} z^{\ell-1}e^{-z}dz\geq K_1 K_2 \bigg(\frac{\lambda}{\Theta (1+ |y| )}\bigg)^\ell, 
	\end{aligned}
	\eee
	where the second line follows from the change of variable $z=\frac{\Theta}{\lambda}(1+ |y|) r$. Combining the above two cases, we conclude from \eqref{eq.lm.ent1} and \eqref{eq.lm.ent} that
	\begin{equation}\label{Gamma esti}
		\Gamma_\lambda(y,i)(u) \leq \max\bigg\{\frac{1}{K_0},   \frac{1}{K_1K_2} \bigg(\frac{\Theta}{\lambda}(1+ |y|)\bigg)^\ell \bigg\}\le C  (1+|y|)^\ell,
	\end{equation}	
	with $C := \max\{\frac{1}{K_0}, \frac{1}{K_1K_2}\left(\frac{\Theta}{\lambda}\right)^\ell\}>0$, which depends on $\iota$, $\vartheta$, $\Theta$, $\lambda$ and $\ell$ (Recall that $K_0$, $K_1$, and $K_2$ depend on $\iota$, $\vartheta$, and $\ell$). It follows that $\ln(\Gamma_\lambda(y,i)(u)) \le \ln C+\ell\ln (1+|y|)$ for all $(y,i,u)\in\R^d\times\Sc\times U$. As $\Gamma_\lambda(y,i)\in\Dc(U)$ by definition, this implies
	\[
	\sup_{i\in\Sc} \int_U \ln(\Gamma_\lambda(y,i)(u)) \Gamma_\lambda(y,i)(u)du \le  \ln C+\ell\ln (1+|y|),\quad \forall y\in \R^d. 
	\] 
	On the other hand, \eqref{eq.kl.lbound} readily shows that $\sup_{i\in\Sc} \int_U \ln(\Gamma_\lambda(y,i)(u)) \Gamma_\lambda(y,i)(u)du \geq -\ln(\Leb(U))$. Hence, in view of \eqref{H}, we conclude that
	\begin{align}
		\sup_{i\in \Sc} \left|\Hc(\Gamma_\lambda(y,i))\right| &= \sup_{i\in \Sc}  \left| \int_U \ln(\Gamma_\lambda(y,i)(u))\Gamma_\lambda(y,i)(u) du \right| \notag \\
		&\leq |\ln(\Leb(U))| +  |\ln C|+\ell\ln (1+|y|),\quad \forall y\in\R^d,  \label{supH esti}
	\end{align}	
	{\color{black}which shows \eqref{eq:supH esti} holds.}
	For $\ell=1$, by following arguments simiar to the above, with $\Delta_\iota\cap B_\vartheta(0)$, $K_0, K_1$, and $K_2$ replaced by $[0,\vartheta]$, $e^{-1}\vartheta, 1$, and $\int_0^1 e^{-z}dz = 1-e^{-1}$, respectively, {\color{black}we can obtain the desired estimate in \eqref{eq:supH esti} for $\ell=1$.}
	
	{\color{black}Now, for $\lambda>0$ small enough, the constant $C>0$ in \eqref{Gamma esti} simply becomes $\frac{1}{K_1K_2}\left(\frac{\Theta}{\lambda}\right)^\ell$ (for both the cases $\ell>1$ and $\ell=1$). The estimate \eqref{supH esti} then directly implies \eqref{eq:supH esti'}. 
}


\subsection{Proof of Theorem~\ref{thm.discrete.relaxedentropy}}\label{subsec:proof of thm.discrete.relaxedentropy}
	First, let us establish the continuity of $\Psi_\lambda:\R^d\to\R^d$. 
	Take an arbitrary sequence $\{y^n\}_{n\in\N}$ in $\R^d$ such that $y^n\to y^\infty$ for some $y^\infty\in\R^d$. 
	As $f(0,i,u)$ and $p^u_i$ are continuous in $u$ (Assumption~\ref{assume.u.lips}), $U$ is compact, and $\Sc$ is a finite set, we conclude from the dominated convergence theorem that 
	\be\label{dct}
	\sup_{i\in \Sc} \bigg| \int_{U} e^{\frac{1}{\lambda}[f(0,i,u)+p^u_i\cdot y^n]} du- \int_{U} e^{\frac{1}{\lambda}[f(0,i,u)+p^u_i\cdot y^\infty]} du \bigg| \to 0,\quad \hbox{as}\ n\to\infty. 
	\ee
	In view of \eqref{Gamma_lambda}, this particularly implies that for any $i\in\Sc$, $\Gamma_\lambda (y^n,i)(u)\to \Gamma_\lambda (y^\infty,i)(u)$ for all $u\in U$. Moreover, for any $i\in\Sc$, observe that
	\begin{align}\label{thm.eq} 
		&\hspace{0.1in}| \Hc\big( \Gamma_\lambda (y^n,i) \big) - \Hc \big( \Gamma_\lambda (y^\infty,i) \big)|\notag\\
		\le &\ \int_U \left| \ln\big(   \Gamma_\lambda (y^n,i)(u)\big) -  \ln\big(   \Gamma_\lambda (y^\infty,i)(u)\big)   \right|  \Gamma_\lambda (y^n,i)(u) du\notag\\
		&\hspace{0.4in}+ \int_U  \left| \ln\big(   \Gamma_\lambda (y^\infty,i)(u)\big) \right|  \big|  \Gamma_\lambda (y^n,i)(u)- \Gamma_\lambda (y^\infty,i)(u) \big|du\notag\\
		\leq &\ \frac1\lambda \left(\sup_{u\in U}|p^u_i|\right) \left| y^n- y^\infty \right| + \left|  \ln\bigg( \int_{U} e^{\frac{1}{\lambda}[f(0,i,v)+p^v_i\cdot y^n]} dv\bigg) - \ln\bigg( \int_{U} e^{\frac{1}{\lambda}[f(0,i,v)+p^v_i\cdot y^\infty]} dv\bigg)   \right|\notag \\
		&\hspace{0.4in}+\sup_{u\in U}\left| \ln\big(   \Gamma_\lambda (y^\infty,i)(u)\big) \right| 
		\int_U  \left|  \Gamma_\lambda (y^n,i)(u)- \Gamma_\lambda (y^\infty,i)(u) \right|du,
	\end{align}
	where the first inequality follows from \eqref{H} and the second inequality is due to \eqref{Gamma_lambda}. As $u\mapsto f(0,i,u)$ and $u\mapsto p^u_i$ are continuous on the compact set $U$, we get the boundedness of $u\mapsto |p^u_i|$ and $u\mapsto \Gamma_\lambda (y^\infty,i)(u)$. Particularly, in view of \eqref{Gamma_lambda}, $u\mapsto \Gamma_\lambda (y^\infty,i)(u)$ is bounded away from zero. This in turn implies that $u\mapsto \ln\big(\Gamma_\lambda (y^\infty,i)(u)\big)$ is bounded. Hence, as $n\to\infty$, we can now conclude from $y^n\to y^\infty$ and \eqref{dct} that the right-hand side of \eqref{thm.eq} vanishes. That is, $\Hc\big( \Gamma_\lambda (y^n,i) \big)\to \Hc\big( \Gamma_\lambda (y^\infty,i) \big)$. 
	
	For any $n\in\N$, 
	consider the Markov chain $\bar X^{n}$ with transition matrix $p^{n}$ given by $p^n_{ij} := \int_U p^u_{ij} \Gamma_\lambda(y^n,i)(u) du$ for all $i,j\in\Sc$, as well as the Markov chain $\bar X^{\infty}$ with transition matrix $p^{\infty}$ given by $p^\infty_{ij} := \int_U p^u_{ij} \Gamma_\lambda(y^\infty,i)(u)du$ for all $i,j\in\Sc$. For each $i\in\Sc$, note that the pointwise convergence $\Gamma_\lambda(y^n,i)\to \Gamma_\lambda(y^\infty,i)$ in $\Dc(U)$ already implies the weak convergence of the corresponding probability measures. This, along with $u\mapsto p^u_{i}$ being continuous for all $i\in\Sc$, yields the convergence of the transition matrices, i.e., $p^n\to p^\infty$ component by component, which in turn implies that the law of $\bar X^n$ converges weakly to that of $\bar X^\infty$. In view of Remark~\ref{rem:same law}, the law of $\bar X^{n}$ (resp.\ $\bar X^{\infty}$) coincides with that of $X^{\Gamma_\lambda(y^n,\cdot)}$ (resp.\ $X^{\Gamma_\lambda(y^\infty,\cdot)}$). That is, we actually have the law of $X^{\Gamma_\lambda(y^n,\cdot)}$ converging weakly to that of $X^{\Gamma_\lambda(y^\infty,\cdot)}$. Now, we can adapt an argument in the proof of \citet[Theorem 3]{HZ21} to our present setting. Specifically, by Skorokhod's representation theorem, there exist $\Sc$-valued processes $Y_n$ and $Y$, defined on some probability space $(\Omega,\mathcal F, P)$, such that the law of $Y_n$ coincides with that of $X^{\Gamma_\lambda(y^n,\cdot)}$, the law of $Y$ coincides with that of $X^{\Gamma_\lambda(y^\infty,\cdot)}$, and $Y^n_k\to Y_k$ for all $k\in\N_0$ $P$-a.s. As $\Sc$ is a finite set, for each $k\in\N_0$, we in fact have $Y^n_k = Y_k$ for $n\in\N$ large enough. It follows that for any $i\in\Sc$,
	\begin{align*}
		&V_\lambda^{\Gamma_\lambda(y^n,\cdot)}(i) = \E_i^P \bigg[ \sum_{k=0}^\infty \left( \int_U f(1+k,Y^n_k,u) \Gamma_\lambda(y^n,Y^n_k) (u)du + \lambda  \delta(1+k)\Hc\big(\Gamma_\lambda(y^n,Y^n_k)\big) \right) \bigg]\\
		&\to  \E_i^P \bigg[ \sum_{k=0}^\infty \left( \int_U f(1+k,Y^\infty_k,u) \Gamma_\lambda(y^\infty,Y^\infty_k) (u)du + \lambda  \delta(1+k)\Hc\big(\Gamma_\lambda(y^\infty,Y^\infty_k)\big) \right) \bigg] = V_\lambda^{\Gamma_\lambda(y^\infty,\cdot)}(i),
	\end{align*}
	where the convergence follows from the dominated convergence theorem, $Y^n_k=Y^\infty_k$ for $n\in\N$ large enough, and $\Gamma_\lambda (y^n,i)\to \Gamma_\lambda (y^\infty,i)$ and $\Hc\big( \Gamma_\lambda (y^n,i) \big)\to \Hc\big( \Gamma_\lambda (y^\infty,i) \big)$ for all $i\in\Sc$. Let us stress that the dominated convergence theorem is applicable here thanks to Assumptions \ref{assume.f.p} and Lemma~\ref{lm.1}. We therefore conclude that $\Psi_\lambda: \R^d \to\R^d$ is continuous.
	
	Now, we are ready to show that a fixed point of $\Psi_\lambda$ exists. For any $y\in \R^d$, 
	\be\label{eq.1'}  
	\begin{aligned}
		|\Psi_\lambda(y)| = \left|V^{\Gamma_\lambda(y,\cdot)}_\lambda(i) \right| &\leq \E_i\bigg[ \sum_{t=0}^\infty \left(\left|f^{\Gamma_\lambda(y,X_t))}(t,i)\right|+ \lambda\delta(t)|\Hc(\Gamma_\lambda(y,X_t))| \right)\bigg]\\
		&\leq   M+\lambda\sum_{t=0}^\infty \delta(t)\phi(|y|) \le \big(1+\lambda\phi(|y|)\big)M,
	\end{aligned}
	\ee
	where the first inequality stems from \eqref{eq.relaxed.V.lambda}, $M>0$ is the constant in Assumption~\ref{assume.f.p}, and the second inequality follows from {\color{black}\eqref{eq:supH esti} in Lemma \ref{lm.1}. Note from \eqref{eq:supH esti}} that $\phi:\R_+\to \R+$ grows sublinearly (i.e., $\phi(\alpha)/\alpha\to 0$ as $\alpha\to\infty$). Hence, $\alpha\mapsto (1+\lambda\phi(\alpha))M$ also grows sublinearly, such that   
	$$
	\alpha^* :=\sup \{ \alpha\geq 0: \alpha\leq  (1+\lambda\phi(\alpha)) M \} <\infty.
	$$
	If $|y|> \alpha^*$, by \eqref{eq.1'} and the definition of $\alpha^*$, $|\Psi_\lambda(y)| \le (1+\lambda\phi(|y|)) M < |y|$. If $|y|\leq \alpha^*$, by \eqref{eq.1'}, $\phi$ being increasing, and the definition of $\alpha^*$, we obtain
	$
	|\Psi_\lambda(y)| \leq M (1+\lambda\phi(|y|)\leq M(1+\lambda\phi(\alpha^*))\leq \alpha^*.
	$
	We then conclude $|\Psi_\lambda(y)|\leq \max\{|y|, \alpha^*\}$ for all $y\in\R^d$. Hence, for any $r\geq \alpha^*$, $\Psi_\lambda(\overline{B_r(0)})\subseteq \overline{B_r(0)}$. As $\Psi_\lambda:\R^d\to\R^d$ is continuous, this implies that $\Psi_\lambda$ has a fixed point $y\in\overline{B_r(0)}$, thanks to Brouwer's fixed-point theorem. By Corollary \ref{coro:Psi}, $\Gamma_\lambda(y,\cdot)\in\Pi_r$ is a {\color{black}regular relaxed equilibrium} for \eqref{eq.relaxed.J}. 


\subsection{Proof of Lemma~\ref{lm.vn.bound}}\label{subsec:proof of lm.vn.bound}
	Note that the existence of the sequence $\{y^n\}_{n\in \N}$ is guaranteed by Theorem~\ref{thm.discrete.relaxedentropy}. 
	{\color{black}By using \eqref{eq:supH esti'} in Lemma \ref{lm.1}} in the calculation \eqref{eq.1'}, we get $|\Psi_\lambda(y)|\le \big(1+\lambda\varphi(\lambda,y)\big)M$ for all $y\in\R^d$, where $\varphi$ is specified in \eqref{eq:supH esti'} and $M>0$ is the constant in Assumption~\ref{assume.f.p}. In view of \eqref{eq:supH esti'},
	\[
	\lambda\varphi(\lambda,y) = \kappa_1\lambda +\kappa_2|\lambda\ln\lambda| +\ell \lambda\ln(1+|y|),
	\]
	where $\kappa_1,\kappa_2>0$ are constants independent of $\lambda$. Since $|\lambda\ln\lambda|\to 0$ as $\lambda\downarrow 0$, the above equation implies that for all $\lambda\in(0,1]$, $\lambda\varphi(\lambda,y)\le \eta(|y|)$, with $\eta(z) := K (1+\ln(1+z))$ for some $K>0$ independent of $\lambda\in (0,1]$. We therefore obtain 
	\begin{equation}\label{Psi bdd ls}
		|\Psi_\lambda(y)|\le \big(1+\eta(|y|)\big)M,\quad\forall y\in\R^d\ \hbox{and}\ \lambda\in(0,1].  
	\end{equation}
	As $\eta:\R_+\to \R+$ grows sublinearly (i.e., $\eta(\alpha)/\alpha\to 0$ as $\alpha\to\infty$), $\alpha\mapsto (1+\eta(\alpha))M$ also grows sublinearly, such that  $\alpha^* :=\sup \{ \alpha\geq 0: \alpha\leq  (1+\eta(\alpha)) M \} <\infty$. By using \eqref{Psi bdd ls} and the definition of $\alpha^*$, we may repeat the argument in the last paragraph of the proof of Theorem~\ref{thm.discrete.relaxedentropy} and obtain 
	\[
	|\Psi_\lambda(y)| 
	\begin{cases}
		< |y|,\quad &\hbox{if}\ |y|>\alpha^*,\\
		\le \alpha^*, &\hbox{if}\ |y|\le\alpha^*,
	\end{cases}
	\quad\forall \lambda\in(0,1].
	\]
	Now, for each $n\in\N$, as $y^n=\Psi_{\lambda_n}(y^n)$, the above inequality entails $|y^n|\le \alpha^*$. This readily implies $\sup_{n\in \N} |y^n|\le \alpha^*<\infty$.


\subsection{Proof of Theorem~\ref{thm.lambdaconvergence.discrete}}\label{subsec:proof of thm.lambdaconvergence.discrete}
	Take any sequence $\{\lambda_n\}_{n\in \N}$ in $(0,1]$ with $\lambda_n\downarrow 0$. By Theorem~\ref{thm.discrete.relaxedentropy}, for each $n\in\N$, there exists $y^n\in \R^d$ such that $\Psi_{\lambda_n}(y^n)=y^n$ and $\pi^n:= \Gamma_{\lambda_n}(y^n,\cdot)\in\Pi_r$ is a {\color{black}regular relaxed equilibrium} for \eqref{eq.relaxed.J} (with $\lambda=\lambda_n$ therein). As the sequence $\{y^n\}_{n\in\N}$ is bounded in $\R^d$ (Lemma \ref{lm.vn.bound}), it has a subsequence (without relabeling) that converges to some $y^\infty\in \R^d$. On the other hand, as $U$ is compact, $\Pc(U)$ is compact under the topology of weak convergence of probability measures. Hence, for each $i\in\Sc$, $\{\pi^n(i)\}_{n\in\N}$ in $\Pc(U)$ has a subsequence (without relabeling) that converges weakly to some $\pi^*(i)\in\Pc(U)$. This gives rise to $\pi^*\in\Pi$ (which is not necessarily regular). 
	
	
	Now, we claim that $\pi^*\in \Gamma(y^\infty)$. In view of \eqref{Gamma}, we need to show that $\pi^*(i)\in\Pc(U)$ is supported by the closed set $E(y^\infty, i)\subseteq U$ in \eqref{eq.discrete.supp} for all $i\in\Sc$. As this holds trivially when $E(y^\infty, i) = U$, we assume $E(y^\infty, i) \subsetneq U$ in the following. Our goal is to prove that for any $i\in\Sc$, $(\pi^*(i))\big(\overline{B_r(u_0)}\big)=0$ for all $u_0\in U$ and $r>0$ such that
	\[
	\overline{B_r(u_0)}\cap E(y^\infty,i)=\emptyset\quad\hbox{and}\quad B_r(u_0)\subseteq U. 
	\]
	Note that this readily implies $\text{supp}(\pi^*(i))\subseteq E(y^\infty,i)$ for all $i\in\Sc$, as desired. To this end, take a continuous and bounded function $h:U\to[0,1]$ such that
	\be\label{h condition}
	h(u)\equiv1\quad \hbox{for}\ u\in B_{r}(u_0)\quad\hbox{and} \quad h(u)\equiv0\quad \hbox{for}\ u\notin \overline{B_{r+d/2}(u_0)}, 
	\ee
	where $d:= \text{dist}\big(\overline{B_r(u_0)}, E(v^\infty,i)\big)$ is strictly positive, as $\overline{B_r(u_0)}$ and $E(v^\infty,i)$ are disjoint closed sets. Consider $A:=\max_{u\in {U} }\left\{f(0,i, u)+ p^u_i\cdot y^\infty \right\}<\infty$. Note that $d>0$ implies	
	\[
	\eps := A-\sup\left\{f(0,i,u)+p^u_i\cdot y^\infty:u\in \overline{B_r(u_0)}\right\} >0. 
	\]
	Also, by the continuity of $u\mapsto f(0,i,u)+p^{u}_i\cdot y^\infty$, 
	\[
	L := \Leb(\{ u\in {U}: f(0,i,u)+p^u_i\cdot y^\infty> A-\eps/2\})>0. 
	\]
	As $y^n\to y^\infty$, for all $n\in\N$ large enough, we have
	$$
	A-\sup\left\{f(0,i,u)+p^u_i\cdot y^n:u\in \overline{B_r(u_0)}\right\}>\frac{\eps}{2},\quad \Leb\bigg(\bigg\{ u\in {U}: f(0,i,u)+p^u_i\cdot y^n > A-\frac{\eps}{2}\bigg\}\bigg)\geq \frac{L}{2}.
	$$ 
	This, along with the definition of $\Gamma_{\lambda_n} (y^n,i)$ in \eqref{Gamma_lambda}, yields
	\begin{align}\label{limsup=0}
		\limsup_{n\to\infty}\int_{U}h(u)\Gamma_{\lambda_n} (y^n,i)(u)du &\leq   \limsup_{n\to\infty} \int_{\overline{B_{r+d/2}(u_0)} } \frac{e^{\frac{1}{\lambda_n}(f(0,i,u)+p^{u}_i\cdot y^n)}}{\int_{U} e^{\frac{1}{\lambda_n}(f(0,i,u)+p^u_i\cdot y^n)} du}du\notag\\
		&=  \limsup_{n\to\infty} \int_{\overline{B_{r+d/2}(u_0)} } \frac{e^{\frac{1}{\lambda_n}(f(0,i,u)+p^{u}_i\cdot y^n-(A-\eps/2))}}{\int_{U} e^{\frac{1}{\lambda_n}(f(0,i,u)+p^u_i\cdot y^n-(A-\eps/2))} du}du \notag\\
		&\leq  \limsup_{n\to\infty} \int_{\overline{B_{r+d/2}(u_0)} } \frac{e^{-\frac{\eps/2}{\lambda_n}}}{\int_{\{ u\in {U}: f(0,i,u)+p^u_i\cdot y^n \geq A-\frac{1}{2}\eps\}} 1 du }du\notag\\
		&\leq   \limsup_{n\to\infty}  \frac2L e^{-\frac{\eps/2}{\lambda_n}}  \Leb\left(\overline{B_{r+d/2}(u_0)}\right)=0. 
	\end{align}
	Now, as $h$ is continuous, bounded, and satisfies \eqref{h condition}, 
	\begin{align*}
		(\pi^*(i))\left(\overline{B_r(u_0)}\right)\leq \int_{U} h(u)(\pi^*(i))(du) &=	\lim_{n\to \infty} \int_{U} h(u) (\pi^n(i))(du)\\
		&=\lim_{n\to \infty} \int_{U} h(u) \Gamma_{\lambda_n}(y^n,i)(u)du=0,
	\end{align*}
	where the first equality follows from $\pi^n(i)\to\pi^*(i)$ weakly in $\Pc(U)$ and the last equality is due to \eqref{limsup=0}. The claim ``$\pi^*\in \Gamma(y^\infty)$'' is then established. 
	
	Now, we set out to prove $V_{\lambda_n}^{\Gamma_{\lambda_n}(y^n,\cdot)}\to V^{\pi^*}$. Consider the Markov chain $\bar{X}^n$ with transition matrix $p^n$ given by $p^n_{ij}:= \int_U p^u_{ij} \Gamma_{\lambda_n}(y^n,i)(u)du$ for all $i,j\in \Sc$, as well as the Markov chain $\bar{X}^*$ with transition matrix $p^*$ given by $p^*_{ij}:= \int_U p^u_{ij} (\pi^*(i))(du)$ for all $i,j\in \Sc$. Similarly to the discussion in the last paragraph of the proof of Theorem~\ref{thm.discrete.relaxedentropy}, $p^n\to p^*$ component by component (thanks to $\Gamma_{\lambda_n}(y^n,i)\to\pi^*(i)$ weakly for all $i\in\Sc$), whence the law of $\bar X^n$ converges weakly to that of $\bar X^*$, which in turn implies that the law of $X^{\Gamma_{\lambda_n}(y^n,\cdot)}$ converges weakly to that of $X^{\pi^*}$ (by Remark~\ref{rem:same law}). By Skorokhod's representation theorem, there exist $\Sc$-valued processes $Y_n$ and $Y$, defined on some probability space $(\Omega,\mathcal F, P)$, such that the law of $Y_n$ coincides with that of $X^{\Gamma_{{\lambda_n}}(y^n,\cdot)}$, the law of $Y$ coincides with that of $X^{\pi^*}$, and $Y^n_k\to Y_k$ for all $k\in\N_0$ $P$-a.s. As $\Sc$ is a finite set, for each $k\in\N_0$, we in fact have $Y^n_k = Y_k$ for $n\in\N$ large enough. It follows that for any $i\in\Sc$,
	\begin{align}
		V_{\lambda_n}^{\Gamma_{\lambda_n}(y^n,\cdot)}&(i) - \E_i^P \bigg[ \sum_{k=0}^\infty \delta(1+k)\lambda_n\Hc\left( \Gamma_\lambda(y^n,Y^n_k)  \right) \bigg]\notag\\
		& =  \E_i^P \bigg[ \sum_{k=0}^\infty \left( \int_U f(1+k,Y^n_k,u) \Gamma_\lambda(y^n,Y^n_k) (u)du  \right) \bigg]\notag\\
		&\to  \E_i^P \bigg[ \sum_{k=0}^\infty \left( \int_U f(1+k,Y_k,u) \pi^*(Y_k) (u)du \right) \bigg]=V^{\pi^*}(i),\label{eq.thm2.2} 
	\end{align}
	where the convergence follows from $Y^n_k = Y_k$ for $n\in\N$ large enough, $\Gamma_\lambda(y^n,i)\to \pi^*(i)$ weakly for all $i\in\Sc$, and $u\mapsto f(t,i,u)$ being continuous on the compact set $U$. Moreover, by \eqref{eq:supH esti'} {\color{black}in Lemma \ref{lm.1}}, the boundedness of $\{y^n\}_{n\in\N}$, and $\sum_{k=0}^\infty \delta(1+k)<\infty$ (Assumption~\ref{assume.f.p}), there exist constants $C_1, C_2>0$ independent of $\{\lambda_n\}_{n\in\N}$ such that 
	\[
	\sum_{k=0}^\infty \delta(1+k)\lambda_n \sup_{i\in \Sc} \left| \Hc(\Gamma_{\lambda_n}(y^n,i)) \right| = \big(C_1\lambda_n + C_2 \lambda_n |\ln \lambda_n|\big) \sum_{k=0}^\infty \delta(1+k)\to 0\quad \text{as }n\to\infty,
	\]
	which implies that $\E_i^P \big[ \sum_{k=0}^\infty \delta(1+k)\lambda_n\Hc\left( \Gamma_\lambda(y^n,Y^n_k)  \right) \big]\to 0$. We then conclude from \eqref{eq.thm2.2} that $V_{\lambda_n}^{\Gamma_{\lambda_n}(y^n,\cdot)}(i) \to V^{\pi^*}(i)$ for all $i\in\Sc$.
	
	Finally, recall that $\Psi_{\lambda_n}(y^n)=y^n$ means $y^n = V_{\lambda_n}^{\Gamma_{\lambda_n}(y^n,\cdot)}$. As a result, 
	\be\label{eq.lambda.1} 
	y^\infty=\lim_{n\to\infty} y^n=\lim_{n\to\infty}V_{\lambda_n}^{\Gamma_{\lambda_n}(y^n,\cdot)}=V^{\pi^*}\in \Psi(y^\infty), 
	\ee 
	where the inclusion follows from $\pi^*\in \Gamma(y^\infty)$. In view of \eqref{Psi} and \eqref{Phi}, the above relation implies $\pi^*\in \Phi(\pi^*)$. Hence, by Proposition \ref{prop.chadiscrete.nonentropy}, $\pi^*$ is a relaxed equilibrium for \eqref{J^pi}. 


\subsection{Proof of Proposition~\ref{propconti.entropy.cha}}\label{subsec:proof of propconti.entropy.cha}
	Fix $\pi\in\Pi_r$. By taking $\pi'=\pi$ in \eqref{eq.lm.dis1} and noting $\tJ^{\pi\otimes_{\eps}\pi}_\lambda(i)=\tJ_\lambda^\pi(i) = \tV^\pi_\lambda(0,i)$, we get
	\be\label{pi'=pi}
	\tV_\lambda^\pi(\eps,i)-\tV_\lambda^\pi(0,i)=-\left(f^{\pi(i)}(0,i)+\lambda \Hc(\pi(i))+Q^{\pi(i)}_i\cdot \tV^\pi_\lambda(\eps)\right)\eps,\quad \forall i\in\Sc.
	\ee
	This implies that $t\mapsto \tV^\pi_\lambda(t,i)$ is continuous. Moreover, when we divide both sides by $\eps>0$ and take $\eps\downarrow 0$, since $t\mapsto \tV^\pi_\lambda(t,i)$ is continuous for all $i\in\Sc$, we get  
	\be\label{eq.lm.1'}
	\partial_t \tV^\pi_\lambda(0,i)+f^{\pi(i)}(0,i)+\lambda  \Hc(\pi(i))+Q^{\pi(i)}_i\cdot \tV^{\pi}_\lambda(0) =0, \quad \forall i\in \Sc.
	\ee 
	Now, for any $\pi'\in \Pi_r$, thanks to \eqref{eq.lm.dis1} and \eqref{pi'=pi}, 
	\begin{align*}
		\tV^{\pi'\otimes_\eps \pi}_\lambda(0,i)- \tV_\lambda^{\pi}(0,i)
		&=\Big([f^{\pi'(i)}(0,i)+\lambda  \Hc(\pi'(i))+Q^{\pi'(i)}_i\cdot \tV^\pi_\lambda(\eps)] \\
		&\hspace{0.5in}- [f^{\pi(i)}(0,i)+\lambda  \Hc(\pi(i))+Q^{\pi(i)}_i\cdot \tV^\pi_\lambda(\eps)] \Big) \eps+o(\eps).
	\end{align*}
	It follows that
	\begin{align}\label{eq.lm.1} 
		&\lim_{\eps \downarrow 0} \frac1\eps \left({\tV_\lambda^{\pi'\otimes_\eps \pi}(0,i) - \tV_\lambda^{\pi}(0,i)}\right)\notag\\
		&= \Big(f^{\pi'(i)}(0,i) + \lambda  \Hc(\pi'(i))+Q^{\pi'(i)}_i\cdot \tV_\lambda^\pi(\eps)\Big)-\Big( f^{\pi(i)}(0,i) +\lambda  \Hc(\pi(i))+Q^{\pi(i)}_i\cdot \tV^\pi_\lambda(\eps) \Big)\notag\\
		&= f^{\pi'(i)}(0,i) + \lambda \Hc(\pi'(i))+Q^{\pi'(i)}_i\cdot \tV^{\pi}_\lambda(0)+\partial_t \tV^{\pi}_\lambda(0,i),\quad \forall i\in \Sc, 
	\end{align}
	where the last line follows from the continuity of $t\mapsto \tV^\pi_\lambda(t,i)$ for all $i\in\Sc$ and \eqref{eq.lm.1'}. Hence, $\pi$ is a {\color{black}regular relaxed equilibrium} for \eqref{eq.relaxed.tJ} if and only if
	$$
	\begin{aligned}
		\partial_t \tV^{\pi}_\lambda(0,i)+\sup_{\rho\in \Dc(U)} \left(f^\rho(0,i)+\lambda \Hc(\rho)+Q^{\rho}_i\cdot \tV^{\pi}_\lambda(0)\right)\le 0,\quad \forall i\in \Sc. 
	\end{aligned}
	$$
	In view of \eqref{eq.lm.1'}, this holds if and only if
	\begin{align}\label{pi=rho^*}
		\pi(i) &\in \argmax_{\rho\in \Dc(U)} \left(f^\rho(0,i)+\lambda \Hc(\rho)+Q^{\rho}_i\cdot \tV^{\pi}_\lambda(0)\right)\notag\\
		& = \argmax_{\rho\in \Dc(U)} \int_{U} \Big( f(0,i,u)-\lambda \ln \rho(u) +q^u_i\cdot \tV^{\pi}_\lambda(0) \Big) \rho(u)du,\quad\forall i\in\Sc.   
	\end{align}
	As the set on the right-hand side above is a singleton that contains the density
	\[
	\rho^*(u) =  \frac{e^{\frac{1}{\lambda} \left(f(0,i,u) +q^u_i\cdot \tV^{\pi}_\lambda(0)\right)} }{\int_{U} e^{\frac{1}{\lambda} \left(f(0,i,v) +q^v_i\cdot \tV^{\pi}_\lambda(0)\right)} dv}=\widetilde\Gamma_\lambda(\tV^{\pi}_\lambda(0),i) (u)=\widetilde\Gamma_\lambda(\tJ^{\pi}_\lambda,i) (u),\quad u\in U,
	\]
	the relation \eqref{pi=rho^*} amounts to $\pi(i) = \widetilde\Gamma_\lambda(\tJ^\pi_\lambda, i)$ for all $i\in\Sc$, which is equivalent to $\pi = \widetilde\Phi_\lambda(\pi)$. 


\subsection{Proof of Lemma~\ref{lm.sec3.1}}\label{subsec:proof of lm.sec3.1}
(a) For notational convenience, we will write $X^{h_n}$ for the discrete-time Markov chain whose transition matrix is $P^n=\{P^n(i)\}_{i\in \Sc}$ with its $i^{th}$-row given by
\[
P^n(i):=\int_U (p_{h_n})^u_i (\pi^n(i))(du)= \int_U \left(h_nq^u_i +e_i\right)(\pi^n(i)) (du)=Q^{\pi^n(i)}_i h_n +e_i,
\]
where the second equality follows from \eqref{eq.phtoQ}. We will also write $X^{n}$ and $X^\infty$ for the continuous-time Markov chains with generators $\{ Q^{\pi^n(i)}_i \}_{i\in\Sc}$ and $\{ Q^{\pi^\infty(i)}_i\}_{i\in\Sc}$, respectively. For each $n\in\N$, note that the transition matrix of the discrete-time Markov chain $\left\{X^n_{k h_n} \right\}_{k\in \N}$ is $\widetilde P^n=\{ \widetilde P^n(i) \}_{i\in \Sc}$ with its $i^{th}$-row given by
\be\label{eq.lm41.0}
\widetilde{P^n}(i) := e_i+Q^{\pi^n(i)}_i h_n+o(h_n)= P^n(i)+ o(h_n).
\ee 
In the following, we will adapt the arguments in the proof of \citet[Lemma 3]{HZ21} to the present setting. For any $i\in\Sc$, observe that
\begin{align}\label{eq.lm31.1}
	&\left|V^{h_n,\pi^n}(i)-\tV^{\pi^\infty}(0,i)\right| \notag\\
	&\leq  h_n \Bigg|  \E_i\left[ \sum_{k=0}^\infty \int_{U} f \left( (k+1)h_n, X_{k}^{h_n}, u \right) \left(\pi^n(X_k^{h_n})\right) (du) \right] \notag\\
	&\qquad \quad  -  \E_i \left[\sum_{k=0}^\infty \int_{U} f\left((k+1)h_n, X^{n}_{k h_n}, u\right) (\pi^n(X^{n}_{k h_n})) (du) \right]\Bigg|\notag\\
	&\quad +\Bigg|   h_n \E_i \left[\sum_{k=0}^\infty \int_{U} f\left((k+1)h_n, X^{n}_{k h_n}, u \right) (\pi^n(X^{n}_{k h_n}) )(du) \right] \notag\\
	& \qquad \quad -  \E_i \left[\int_0^\infty \int_{U} f(t, X^{n}_{t}, u) (\pi^n(X^{n}_{t}))(du) dt \right]\Bigg|\notag\\
	&\quad + \left|   \E_i \left[\int_0^\infty \int_{U} f(t, X^{n}_{t}, u) (\pi^n(X^{n}_{t}))(du) dt \right]- \E_i \left[\int_0^\infty  \int_{U} f(t, X^{\infty}_{t}, u) (\pi^\infty(X^{\infty}_{t}))( du) dt \right]\right|.
\end{align}
Let $I_1^n$, $I_2^n$, and $I_3^n$ denote the second, third, and fourth lines, respectively, in the above inequality. 

Let us first deal with $I_1^n$. By Assumption \ref{assume.f.Q}, for any $\eps>0$, we can take $T>0$ such that 
\begin{equation}\label{T large}
	\int_T^\infty \sup_{i, u} |f(t, i,u)|dt<\eps. 
\end{equation}
It follows that
\be\label{eq.I1.0} 
\begin{aligned}
	I^n_1&\leq h_n \bigg|  \sum_{k=0}^{T/h_n} \E_i\bigg[ \int_{U} f\left( (k+1)h_n, X_{k}^{h_n}, u  \right) \left( \pi^n(X_k^{h_n}) \right) (du) \\
	&\hspace{1.2in}- \int_U f\left( (k+1)h_n, X^{n}_{k h_n}, u \right) \left( \pi^n(X^{n}_{k h_n}) \right) (du) \bigg] \bigg| + 2\eps
\end{aligned}
\ee 
In addition, \eqref{eq.lm41.0} implies $(\widetilde{P^n})^k(i)= (P^n)^k(i)+ k o(h_n)(1+o(h_n))^k$, where $(\widetilde{P^n})^k(i)$ (resp.\ $(P^n)^k(i)$) denotes the $i^{th}$-column of the matrix $(\widetilde{P^n})^k$ (resp.\ $(P^n)^k$). By writing $\int_{U}f(kh_n, \cdot , u)(\pi^n(\cdot))(du)$ for the vector $\left( \int_{U}f(kh_n, 1 , u)(\pi^n(1))(du),...,\int_{U}f(kh_n, d , u)(\pi^n(d))(du) \right)\in\R^d$, we observe that
\be\label{eq.I1.1}  
\begin{aligned}
	\E_i\left[\int_{U} f \left( (k+1)h_n, X^{h_n}_k, u \right) \pi^n\left( X^{h_n}_k \right) (du)\right] &=(P^n)^k(i) \cdot  \left(\int_{U}f\left((k+1)h_n, \cdot , u \right)(\pi^n(\cdot))(du)\right)\\
	\E_i\left[\int_{U} f\left( (k+1)h_n, X_{kh_n}^{n}, u \right) \pi^n\left( X^n_{k h_n} \right)(du)\right] &= (\widetilde{P^n})^k(i) \cdot \left(\int_{U}f((k+1)h_n, \cdot , u)(\pi^n(\cdot))(du)\right).
\end{aligned}
\ee 
It then follows from that
\begin{align*}
	I^n_1\leq h_n \widetilde M\sum_{k=0}^{T/h_n} k o(h_n)(1+o(h_n))^k +2\eps &\le h_n \widetilde M\sum_{k=0}^{T/h_n} \frac{T}{h_n} o(h_n) (1+h_n)^{T/h_n} +2\eps\\
	& = \sum_{k=0}^{T/h_n}  o(h_n) = \bigg(\frac{T}{h_n} +1\bigg) o(h_n)+2\eps = o(1)+2\eps. 
\end{align*}
We then obtain $\lim_{n\to\infty} I^n_1 \le 2\eps$. As $\eps>0$ is arbitrary, we conclude $\lim_{n\to\infty} I^n_1=0$. 

We now deal with $I^n_2$. For any $\eps>0$, consider $T>0$ as in \eqref{T large}. Then, observe that $I^n_2\leq \sum_{k=0}^{T/h_n} \E_i[\eta_k] +2\eps$, with 
\[
\eta_k := \left| h_n \int_{U} f \left( (k+1)h_n, X^{n}_{k h_n}, u \right) (\pi^n(X^{n}_{k h_n}) )(du)  -   \int_{kh_n}^{(k+1)h_n}\int_{U} f(t, X^{n}_{t}, u) (\pi^n(X^{n}_{t}))(du) dt\right|.
\]
Set $A_k:= \{ \text{there is no jump for $X^n$ in the time interval $(k h_n, (k+1)h_n]$}\}$. As $f(\cdot,i,\cdot)$ is continuous, it is uniformly continuous on the compact set $[0,T]\times U$. Hence, there exists a modulus of continuity $L$, independent of $i$ and $u$, such that $|f(t,i,u)-f(s,i,u)|\leq L(|t-s|)$ for all $t, s\in[0,T]$. It follows that 
$$
\E_i[\eta_k]\leq \E_i[\eta_k\mid A_k] \P(A_k) + \E_i[\eta_k\mid A^c_k] \P(A^c_k)\leq L(h_n) h_n (1-o(1))+O(h_n)o(1)=o(h_n).
$$
Hence, $I^n_2\leq \sum_{k=0}^{T/h_n} o(h_n)+2\eps=o(1)+2\eps$, which implies $\lim_{n\to\infty} I^n_2\le 2\eps$. As $\eps>0$ is arbitrary, we conclude $\lim_{n\to\infty} I^n_2=0$. 

Finally, we deal with $I^n_3$. For any $i\in\Sc$, as $u\mapsto q^i_u$ is continuous and $U$ is compact, the fact ``$\pi^n(i)\to \pi^\infty(i)$ weakly'' readily implies $Q^{\pi_n(i)}_i\to Q^{\pi^\infty(i)}_i$. That is, the rate matrix of $X^n$ converges to that of $X^\infty$ (component by component). Then, we may follow the argument in the proof of \citet[Theorem 3]{HZ21} (particularly, from the third last line of p.\ 448 to the fourth line on p.\ 449) to obtain $\lim_{n\to\infty} I^n_3= 0$. As $I^n_1$, $I^n_2$, and $I^n_3$ all converge to zero,  we conclude from \eqref{eq.lm31.1} that $V^{h_n,\pi^n}(i)\to \tV^{\pi^\infty}(0,i)$.  

(b) For any $i\in\Sc$ and $u\in U$, thanks to \eqref{Gamma_lambda}, \eqref{f_h}, and \eqref{eq.phtoQ}, 
\begin{align}\label{eq.thm3.1}
	\Gamma_{h_n \lambda}(y^n,i)(u) &=\frac{\exp[\frac{1}{h_n\lambda}(f_{h_n}(0,i,u) +(p_h)^u_i \cdot y^n)]}{\int_{U} \exp[\frac{1}{h_n \lambda}(f_{h_n}(0,i,u) +(p_h)^u_i\cdot y^n )]du} =  \frac{\exp[\frac{1}{\lambda}(f(0,i,u) +q^u_i\cdot y^n)]}{\int_{U} \exp[\frac{1}{\lambda}(f(0,i,u) +q^u_i \cdot y^n)]du}\notag\\
	&\to  \frac{\exp \left[\frac{1}{\lambda}(f(0,i,u) +q^u_i\cdot y^\infty ) \right]}{\int_{U} \exp[\frac{1}{\lambda}(f(0,i,u) +q^u_i \cdot y^\infty )]du} = \widetilde\Gamma_\lambda(y^\infty,i)(u),\quad \hbox{as}\ n\to\infty, 
\end{align}
where the convergence follows from $y^n\to y^\infty$. 
In view of \eqref{tGamma_lambda}, the above implies $\Gamma_{h_n \lambda}(y^n,i)(u) = \widetilde\Gamma_\lambda (y^n,i)(u)\to \widetilde\Gamma_\lambda(y^\infty,i)(u)$. Hence, $\Hc(\Gamma_{h_n\lambda} (y^n,i)) = \Hc(\widetilde\Gamma_\lambda (y^n,i)) \to  \Hc(\widetilde\Gamma_{\lambda} (y^\infty,i))$, where the convergence follows from an argument similar to \eqref{thm.eq}. 
By taking $\pi^n:=\Gamma_{h_n\lambda}(y^n,\cdot)$, the desired result follows from the arguments in part (a) and $\Hc(\Gamma_{h_n\lambda} (y^n,i))  \to  \Hc(\widetilde\Gamma_{\lambda} (y^\infty,i))$ for all $i\in\Sc$.


\subsection{Proof of Theorem~\ref{thm.discts.convergence}}\label{subsec:proof of thm.discts.convergence}
(a) For any $n\in\N$, set $y^n:= V^{h_n, \pi^n}_{\lambda}\in\R^d$. As $\pi^n\in\Pi_r$ is a {\color{black}regular relaxed equilibrium} for  \eqref{eq.relaxed.J^h}, Proposition \ref{propdiscr.entropy.cha} implies $\pi^n=\Phi_{h_n \lambda}(\pi^n)$, i.e., $\pi^n(i) = \Gamma_{h_n \lambda}(y^n,i)$. In addition, Corollary~\ref{coro:Psi} implies $\Psi_{h_n\lambda}(y^n ) = y^n$. Hence, by Lemma~\ref{lm.vn.bound}, $\{y^n\}_{n\in\N}$ is bounded in $\R^d$. For any subsequence of $\{y^n\}_{n\in\N}$ (without relabeling) that converges to some $y^\infty\in\R^d$, Lemma \ref{lm.sec3.1} (b) asserts $\Gamma_{h_n \lambda}(y^n,i)\to \Gamma_{\lambda}(y^\infty,i)$. Thus, $\pi^\infty(i) := \lim_{n\to\infty} \pi^n(i) = \widetilde\Gamma_\lambda(y^\infty,i)$ is well-defined for all $i\in\Sc$. Now, note that
\[
y^\infty = \lim_{n\to\infty} y^n = \lim_{n\to\infty} V^{h_n,\pi^n}_{\lambda} = \lim_{n\to\infty}  V^{h_n, \Gamma_{h_n\lambda}(y^n,\cdot)}_{\lambda} = \tJ_\lambda^{\widetilde\Gamma_\lambda(y^\infty,\cdot)},
\]	
where the last equality follows from Lemma \ref{lm.sec3.1} (b). That is, we have $y^\infty= \widetilde\Psi_\lambda(y^\infty)$. By Corollary~\ref{coro:Psi'}, this implies $\pi^\infty= \widetilde\Gamma_\lambda(y^\infty,\cdot)$ is a {\color{black}regular relaxed equilibrium} for \eqref{eq.relaxed.tJ}. 

(b) As $U$ is compact, $\Pc(U)$ is compact under the topology of weak convergence of probability measures. Hence, for each $i\in\Sc$, $\{\pi^n(i)\}_{n\in\N}$ in $\Pc(U)$ has a subsequence (without relabeling) that converges weakly to some $\pi^*(i)\in\Pc(U)$. By Proposition \ref{prop.chadiscrete.nonentropy}, for each $n\in\N$,
\begin{align}\label{pi^n supp}
	\operatorname{supp}(\pi^n(i)) &\subseteq  \argmax_{u\in {U}} \left\{  f_h(0,i,u) + (p_h)^u_i \cdot V^{h_n, \pi^n} \right\} \notag \\
	&=  \argmax_{u\in {U}} \left\{ f(0,i,u) + q^u_i \cdot V^{h_n, \pi^n} \right\} =	\widetilde{E} \left( V^{h_n, \pi^n},i \right),\quad \forall i\in \Sc,
\end{align}
where the first equality follows from \eqref{f_h} and \eqref{eq.phtoQ} and the last equality holds in view of \eqref{eq.conti.supp}.  
By Proposition \ref{prop.nonentropy.cha}, to show that $\pi^\infty$ is a relaxed equilibrium, it suffices to prove $\pi^\infty\in \widetilde{\Phi}(\pi^\infty)$, i.e., $\operatorname{supp}(\pi^\infty(i))\subseteq  \widetilde{E}(\tJ^{\pi^\infty}, i)$ for all $i\in\Sc$. By contradiction, suppose that for some $i\in\Sc$, there exist $u_0\in U$ and $r>0$ such that
\be\label{eq.thm3.1'} 
\operatorname{dist} \left( B_r(u_0),  \widetilde{E}(\tJ^{\pi^\infty},i) \right)>0\quad\hbox{and}\quad (\pi^\infty(i))(B_r(u_0)) >0
\ee 
By the weak convergence of $\pi^n(i)$ to $\pi^\infty(i)$, $\liminf_{n\to\infty}(\pi^n(i))(B_r(u_0)) \ge (\pi^\infty(i))(B_r(u_0))>0$. This, along with \eqref{pi^n supp}, implies $B_r(u_0)\cap \widetilde{E}(V^{h_n, \pi^n},i)\neq \emptyset$ for $n\in\N$ large enough. Take $u_n\in B_r(u_0)\cap \widetilde{E}(V^{h_n, \pi^n},i)\in U$ for $n\in\N$ large enough. As $\{u_n\}_{n\in\N}$ is a bounded sequence, it converges up to a subsequence to some $u^*\in \overline{B_r(u_0)}\cap U$. Now, by Lemma \ref{lm.sec3.1} (a),
\begin{align*}
	f(0,i,u^*)+ q^{u^*}_i \cdot \tJ^{\pi^\infty}&=\lim_{n\to\infty} \left\{ f(0,i,u_n)+ q^{u_n}_i \cdot V^{h_n, \pi^n} \right\}\\
	&=\lim_{n\to\infty} \max_{u\in U} \left\{f(0,i,u) + q^u_i \cdot V^{h_n, \pi^n} \right\}\ge \max_{u\in U} \left\{f(0,i,u)+q^u_i \cdot\tJ^{\pi^\infty} \right\},
\end{align*}
where the second equality follows from $u_n\in \widetilde{E}(V^{h_n,\pi^n},i)$ and the inequality holds by exchanging the limit and maximization. As $u^*\in U$, the above inequality is in fact an equality, which implies $u^*\in \widetilde{E}(\tJ^{\pi^\infty},i)$. The fact $u^*\in \overline{B_r(u_0)}\cap  \widetilde{E}(\tJ^{\pi^\infty},i)$ readily contradicts the first condition in \eqref{eq.thm3.1'}. 

\bibliographystyle{agsm}
\bibliography{referencenew}

\end{document}